\documentclass[12pt]{amsart}

\usepackage{latexsym}
\usepackage{amssymb}
\usepackage{amsmath}
\usepackage{amsfonts}

\usepackage{setspace,  
hyperref, epigraph} 

\usepackage{breakurl}   

\usepackage{enumitem} 

\usepackage{tikz}
\usetikzlibrary{decorations.pathreplacing}

\newtheorem{theorem}{Theorem}[section]
\newtheorem{lemma}[theorem]{Lemma}         
\newtheorem{proposition}[theorem]{Proposition}    
\newtheorem{corollary}[theorem]{Corollary}         

\theoremstyle{definition}
\newtheorem{definition}[theorem]{Definition}   
\newtheorem{example}[theorem]{Example}

\newcommand{\bbR}{\mathbb{R}}

\newcommand{\bbV}{\mathbb{V}}
\newcommand{\bbF}{\mathbb{F}}

\newcommand{\bbH}{\mathbb{H}}

\newcommand{\bbP}{\mathbb{P}}
\newcommand{\bbZ}{\mathbb{Z}}
\newcommand{\bbN}{\mathbb{N}}

\newcommand{\bbQ}{\mathbb{Q}}

\newcommand{\bbT}{\mathbb{T}}

\newcommand{\sM}{M^{\ast}}

\newcommand{\cM}{\mathfrak{M}}
\newcommand{\scM}{\mathfrak{M}^{\ast}}
\newcommand{\cU}{\mathcal{U}}

\newcommand{\cP}{\mathcal{P}}

\newcommand{\cD}{\mathcal{D}}
\newcommand{\cG}{\mathcal{G}}
\newcommand{\cT}{\mathcal{T}}

\newcommand{\fN}{\mathfrak{N}}

\newcommand{\fL}{\mathfrak{L}}

\newcommand{\imp}{\rightarrow}
\newcommand{\eqi}{\longleftrightarrow}
\newcommand{\ra}{\rangle}
\newcommand{\la}{\langle}
\newcommand{\uh}{\!\upharpoonright\!}

\newcommand{\st}{\operatorname{\mathbf{st}}}
\newcommand{\insl}{\operatorname{\mathbf{in}}}
\newcommand{\fin}{\operatorname{\mathbf{fin}}}
\newcommand{\dom}{\operatorname{dom}}

\newcommand{\rank}{\operatorname{rank}}
\newcommand{\infy}{\operatorname{\mathbf{aa}}}
\newcommand{\sh}{\operatorname{\mathbf{sh}}}

\newcommand{\ZF}{\mathbf{ZF}}
\newcommand{\ZFC}{\mathbf{ZFC}}
\newcommand{\ZFc}{\mathbf{ZF}c}
\newcommand{\SPOT}{\mathbf{SPOT}}
\newcommand{\BSPT}{\mathbf{BSPT}}
\newcommand{\BSCT}{\mathbf{BSCT}}
\newcommand{\ISPT}{\mathbf{ISPT}}
\newcommand{\SCOT}{\mathbf{SCOT}}

\newcommand{\ISCT}{\mathbf{ISCT}}
\newcommand{\BST}{\mathbf{BST}}
\newcommand{\IST}{\mathbf{IST}}
\newcommand{\HST}{\mathbf{HST}}

\newcommand{\T}{\mathbf{T}}
\newcommand{\B}{\mathbf{B}}
\newcommand{\BI}{\mathbf{BI}}

\newcommand{\ACC}{\mathbf{ACC}}
\newcommand{\ADC}{\mathbf{ADC}}
\newcommand{\SP}{\mathbf{SP}}
\newcommand{\SC}{\mathbf{SC}}
\newcommand{\SN}{\mathbf{SN}}
\newcommand{\N}{\mathbf{O}}
\newcommand{\AC}{\mathbf{AC}}
\newcommand{\CH}{\mathbf{CH}}
\newcommand{\CC}{\mathbf{CC}}
\newcommand{\UP}{\mathbf{UP}}

\newcommand{\pmba}{\raisebox{-6pt}{\begin{tikzpicture}[scale=.3]
\draw [dashed, decorate, decoration={brace, amplitude=5pt}] (0,0) -- (0,2);
\end{tikzpicture}}}

\newcommand{\pmbb}{\raisebox{-6pt}{\begin{tikzpicture}[scale=.3]
\draw [dashed, decorate, decoration={brace, amplitude=5pt}] (2,2) -- (2,0);
\end{tikzpicture}}}

\newcommand{\pmbaa}{\pmba\hspace{-5pt}\pmba}

\newcommand{\pmbbb}{\pmbb\hspace{-5pt}\pmbb}

\newcommand{\deltast}{\mathbf{\Delta\!}^{\st}}


\begin{document}\title {Infinitesimal analysis without the Axiom of Choice}
\author{Karel Hrbacek}
\address{Department of Mathematics\\                City College of CUNY\\                New York, NY 10031\\}                \email{khrbacek@icloud.com}

\author{Mikhail G. Katz}
\address{Department of Mathematics\\    Bar Ilan University\\       Ramat Gan 5290002 Israel\\}                \email{katzmik@math.biu.ac.il}

\date{February 9, 2021}

\begin{abstract}
It is often claimed that analysis with infinitesimals requires more
substantial use of the Axiom of Choice than traditional elementary
analysis.  The claim is based on the observation that the hyperreals
entail the existence of nonprincipal ultrafilters over $\bbN$, a
strong version of the Axiom of Choice, while the real numbers can be
constructed in $\ZF$.  The axiomatic approach to nonstandard methods
refutes this objection.  We formulate a theory $\SPOT$ in the
$\st$-$\in$-language which suffices to carry out infinitesimal
arguments, and prove that $\SPOT$ is a conservative extension of
$\ZF$. Thus the methods of Calculus with infinitesimals are just as
effective as those of traditional Calculus.  The
conclusion extends to large parts of ordinary mathematics and beyond.
We also develop a stronger axiomatic system $\SCOT$, conservative over
$\ZF+\ADC$, which is suitable for handling such features as an
infinitesimal approach to the Lebesgue measure.  Proofs of the
conservativity results combine and extend the methods of forcing
developed by Enayat and Spector.
\end{abstract}

\maketitle

\tableofcontents

\section{Introduction} 
\label{Intro}
Many branches of mathematics exploit the Axiom of Choice ($\AC$) to
one extent or another.  It is of considerable interest to gauge how
much the Axiom of Choice can be weakened in the foundations of
nonstandard analysis.  Critics of analysis with infinitesimals
often claim that nonstandard methods require more substantial use of
$\AC$ than their standard counterparts. The goal of this paper is to
refute such a claim.

\subsection{Axiom of Choice in Mathematics}\label{axofchoice}
We begin by considering the extent to which $\AC$ is needed in traditional non-infinitesimal mathematics.
Simpson~\cite{Si} introduces a useful distinction between \emph{set-theoretic} mathematics  and \emph{ordinary} or \emph{non-set-theoretic} mathematics. The former includes such disciplines as general topology, abstract algebra and  functional analysis. It is well known that fundamental theorems in these areas require strong versions of $\AC$. Thus 
\begin{itemize}
\item
Tychonoff's Theorem in general topology is equivalent to full $\AC$
(over Zermelo-Fraenkel set theory $\ZF$).
\item
Prime Ideal Theorem asserts that every ring with unit has a (two-sided) prime ideal.  $\textbf{PIT}$ is an essential result in abstract algebra and is ``almost'' as strong as $\AC$ (it is equivalent over $\ZF$ to Tychonoff's Theorem for Hausdorff spaces). It is also equivalent to the Ultrafilter Theorem: Every proper filter over a set $S$  (i.e., in $\cP(S)$) can be extended to an ultrafilter.
\item
Hahn-Banach Theorem for general vector spaces is equivalent to the statement that every Boolean algebra admits a real-valued measure, a form of $\AC$ that is somewhat weaker than $\textbf{PIT}$. 
\end{itemize} 
Jech~\cite{J1} and Howard and Rubin~\cite{Ho} are comprehensive references for the relationships between these and many other forms of $\AC$.

Researchers in set-theoretic mathematics have to accept strong forms of $\AC$ as legitimate whether or not they use nonstandard methods. 
Our concern here is with ordinary mathematics, which, according to Simpson, includes fields such as the Calculus, countable algebra, differential equations, and real and complex analysis. It is often felt that  results in these fields should be effective in the sense of not being dependent on $\AC$. 
However, even these branches of mathematics cannot do entirely without $\AC$. There is a number of fundamental classical results that rely on it; they include 
\begin{itemize}
\item
the equivalence of continuity and sequential continuity for real-valued functions on $\bbR$;
\item
 the equivalence of the $\varepsilon$-$\delta$ definition and the sequential definition of  closure points for subsets of $\bbR$;
\item
closure of the collection of Borel sets under countable unions and intersections;
\item
countable additivity of Lebesgue measure. 
\end{itemize}
Without an appeal to $\AC$ one cannot even prove that $\bbR$ is not a union of countably many countable sets, or that a strictly positive
function cannot have vanishing Lebesgue integral (Kanovei and Katz~\cite{KK}).  However, these results follow already from $\ACC$, the Axiom of Choice for Countable collections, a weak version of $\AC$ that many  mathematicians use without even noticing.\footnote{For example Halmos~\cite{Ha}, p.\;42; see~\cite{Katz1}, Sec.\;5.7 for further discussion.}
Nevertheless, it is true enough that no choice is needed to define the real number system itself, or to develop the Calculus and much of ordinary mathematics.

It has to be emphasized that objections to $\AC$ are not a matter of ontology, but of epistemology.
In other words, the issue is not the existence of objects, but proof techniques and procedures.
For better or worse, many mathematicians nowadays believe that the objects of interest to them can be represented by set-theoretic structures in a universe that satisfies $\ZFC$, Zermelo-Fraenkel set theory with the Axiom of Choice. 
Nevertheless,  they may prefer results that are effective, that is, do not use $\AC$. For the purposes of this discussion, mathematical results are \emph{effective} if they can be proved in $\ZF$.  Much of ordinary mathematics is effective in this sense. 

We now consider whether nonstandard methods require anything more.
A common objection to infinitesimal methods in the Calculus is the claim that the mere existence of the hyperreals%
\footnote{By the hyperreals we mean a proper elementary extension of
the reals, i.e., a proper extension that satisfies Transfer.  The
definite article is used merely for grammatical correctness.  Subsets
of $\bbN$ can be identified with real numbers; see $\SP \Rightarrow
\SP'$ in the proof of Lemma~\ref{spandsp'} for one way to do that.}
implies the existence of a nonprincipal ultrafilter $\cU$ over $\bbN$. 
The proof is simple: Fix an infinitely large integer $\nu$ in ${}^{\ast}\bbN \setminus \bbN$ and define $\cU \subseteq \cP(\bbN)$ by
$X \in \cU \eqi \nu \in  {}^{\ast}\!X$, for $X \subseteq \bbN$. 
It is easy to see that $\cU $ is a nonprincipal ultrafilter over $\bbN$.
For example, if $X \cup Y \in \cU$, then $\nu \in {}^{\ast} (X \cup Y) = {}^{\ast}\!X \cup \,{}^{\ast}Y$, where the last step is by the Transfer Principle. Hence either  $\nu \in {}^{\ast}\!X$ or  $\nu \in {}^{\ast}Y$, and so $X \in \cU $ or $Y \in \cU$. If $X$ is finite, then $X =  {}^{\ast}\!X$, hence $\nu \notin  {}^{\ast}\!X$ and so $\cU $ is nonprincipal.

By the well-known result of Sierpi\'{n}ski~\cite{Sp} (see also Jech~\cite{J1}, Problem 1.10), $\cU$ is a non-Lebesgue-measurable set (when subsets of $\bbN$ are identified with real numbers in some natural way). In the celebrated model of Solovay~\cite{So}, $\ZF +\ACC$ holds (even the stronger $\ADC$, the Axiom of Dependent Choice, holds there), but all sets of real numbers are Lebesgue measurable, hence there are no nonprincipal ultrafilters over $\bbN$ in this model. The existence of nonprincipal ultrafilters over $\bbN$ requires a strong version of $\AC$ such as $\textbf{PIT}$; it cannot be proved in $\ZF$ (or even $\ZF +\ADC$).

\subsection{Countering the objection}

How can such an objection be answered? As in the case of the
traditional mathematics, the key is to look not at the objects but at
the methods used.  
Currently there are two popular ways to practice Robinson's nonstandard analysis: the
model-theoretic approach and the axiomatic/syntactic approach.
Analysis with infinitesimals does not have to be
based on hyperreal structures in the universe of $\ZFC$. It can be
developed axiomatically; the monograph by Kanovei and Reeken~\cite{KR}
is a comprehensive reference for such approaches. Internal axiomatic
presentations of nonstandard analysis, such as $\textbf{IST}$ or
$\textbf{BST}$, extend the usual $\in$-language of set theory by a
unary predicate $\st$ ($\st(x)$ reads $x$ \emph{is standard}).  For
reference, the axioms of $\BST$ are stated in
Section~\ref{idealization}.

It is of course possible to weaken $\ZFC$ to $\ZF$ within $\BST$ or $\IST$, but this move by itself does not answer the above objection.
It is easily seen, by a modification of the argument given above for hyperreals, that the theory obtained from $\textbf{BST}$ or $\IST$ by replacing $\ZFC$ by $\ZF$ proves $\textbf{PIT}$ (Hrbacek~\cite{H1}).
This argument uses the full strength of the principles of Idealization and Standardization (see Section~\ref{idealization}).
However, Calculus with infinitesimals can be fully carried out assuming much less.
Examination of texts such as Keisler~\cite{K1} and Stroyan~\cite{St} reveals that only very weak versions of these principles are ever used there. Of course one has to postulate that infinitesimals exist (Nontriviality), but stronger consequences of Idealization are not needed. As for Standardization, these textbooks only explicitly postulate a special consequence of it, namely, the following principle:
\begin{enumerate}
\item[$\SP$] (\emph{Standard Part}) Every limited real is infinitely
close to a standard real;
\end{enumerate}
see Keisler~\cite{K1,K2}, Axioms A - E.\, However, this is somewhat misleading. Keisler does not develop the Calculus from his axioms alone; they describe some properties of the hyperreals, but the hyperreals are considered to be an extension of the field $\bbR$ of real numbers  in the universe of $\ZFC$, and the principles of $\ZFC$ can be freely used. In particular, the principle of Standardization is not an issue;  it is automatically satisfied for any formula. While Standardization for  formulas about integers appears innocuous, Standardization for formulas about reals can lead to the existence of nonprincipal  ultrafilters. 
On the other hand, some instances of Standardization over the reals are unavoidable, for example to prove the existence of the function $f'$ (the derivative of $f$) defined in terms of infinitesimals for a given real-valued function $f$ on~$\bbR$.

\subsection{$\SPOT$  and $\SCOT$}
\label{s13}

In the present text, we introduce a theory $\SPOT$ in the
$\st$-$\in$-language, a subtheory of $\textbf{IST}$ and
$\textbf{BST}$, and we show that $\SPOT$ proves Countable Idealization
and enough Standardization for the purposes of the Calculus.  We use
$\forall$ and $\exists$ as quantifiers over sets and $\forall^{\st}$
and $\exists^{\st}$ as quantifiers over standard sets.  The axioms of
$\SPOT$ are:

\bigskip
$\ZF$ (Zermelo - Fraenkel Set Theory)

\bigskip
$\T$ (Transfer) 
Let $\phi$ be an $\in$-formula with standard parameters. Then
$$\forall^{\st} x\; \phi(x)  \imp  \forall x\; \phi(x).$$

$\N$ (Nontriviality)  \quad
$\exists \nu \in \bbN\; \forall^{\st} n \in \bbN\; (n \ne \nu)$.

\bigskip
$\SP'$  (Standard Part)
$$ \forall  A \subseteq \bbN \; \exists^{\st} B \subseteq \bbN \; \forall^{\st}  n \in \bbN \;
(n \in B \eqi n \in A). $$

Our main result is the following.

\bigskip
\textbf{Theorem\;A}\;
\emph{The theory\, $\SPOT$ is a conservative extension of\, $\ZF$.}

\bigskip
Thus the methods used in the Calculus with infinitesimals do not require any appeal to the Axiom of Choice.

The result allows significant strengthenings. We let $\SN$ be the
Standardization principle for $\st$-$\in$-formulas with no
parameters (see  Section~\ref{forcingexternal}). The principle allows Standardization of much more complex
formulas than $\SPOT$ alone.

\bigskip
\textbf{Theorem B}\;
\emph{The theory\, $\SPOT + \SN$ is a conservative extension of $\ZF$.} 

\bigskip
It is also possible to add some Idealization. We let $\BI'$ be Bounded
Idealization (see Section~\ref{idealization}) for $\in$-formulas with
standard parameters.

\bigskip
\textbf{Theorem C}\; \emph{The theory\, $\SPOT + \B + \BI'$ is a conservative extension of $\ZF$. } 

\bigskip
This is  the theory $\BST$ with $\ZFC$ replaced by $\ZF$, Standardization weakened to $\SP$ and Bounded Idealization weakened to $\BI'$; we denote it $\BSPT'$.\label{bspt'}
This theory enables the applicability of some infinitesimal techniques to arbitrary topological spaces.  It also  proves that there is a finite set $S$ containing all standard reals, a frequently used idea.

As noted above, some important results in elementary analysis and elsewhere in ordinary mathematics require the Axiom of Countable Choice.  On the other hand, $\ACC$ entails no ``paradoxical'' consequences,  
such as the existence of Lebesgue-non-measurable sets, or the existence of an additive function on $\bbR$ different from  $f_a \colon x \mapsto ax$ for all $a \in \bbR$.
Many mathematicians find $\ACC$ acceptable. These considerations apply as well to the following stronger 
axiom.

\bigskip
$\ADC$ (Axiom of Dependent Choice)
If $R$ is a binary relation on a  set $A$ such that $\forall a \in A\; \exists a' \in A \; (a Ra')$, then  for every $a \in A$ there exists a sequence $\la a_n \,\mid\, n \in \bbN \ra$ such that $a_0 = a$ and $ a_n R a_{n+1} $ for all $n \in \bbN$. 

\bigskip
This axiom is needed for example to prove the equivalence of the two definitions of a well-ordering
(Jech~\cite{J2}, Lemma 5.2):

(1) Every nonempty subset of a linearly ordered set $(A, <)$ has a least element. 

(2) $A$ has no infinite decreasing sequence $a_0 > a_1 > \ldots > a_n > \ldots$. 

\bigskip
We denote by $\ZFc$ (``$\ZFC$ Lite'') the theory  $ \ZF + \ADC$. This theory is sufficient for axiomatizing ordinary mathematics (and many results of set-theoretic mathematics as well). 
Let $\SCOT$ be the theory obtained from $\SPOT$ by  strengthening $\ZF$ to $\ZFc$, $\SP$ to Countable
$\st$-$\in$-Choice  ($\CC$), and adding $\SN$; see Section~\ref{scot}.

\bigskip
\textbf{Theorem D}
\;The theory $\SCOT $ is a conservative extension of $\ZFc$. 

\bigskip
In $\SCOT$ one can carry out most techniques used in infinitesimal treatments of ordinary mathematics.
As examples, we give a proof of Peano's Existence Theorem and an infinitesimal construction of Lebesgue measure  in Section~\ref{scot}.
Thus the nonstandard methods used in ordinary mathematics do not require any more choice than is generally accepted in traditional ordinary mathematics.

Further related conservative extension theorems can be found in Sections~\ref{stand},~\ref{forcingexternal} and~\ref{idealization}.

\begin{table}
\def\drawing {
\[
\begin{tabular}[t]{|
@{\hspace{3pt}}p{.8in}|| 
@{\hspace{3pt}}p{.7in}|
@{\hspace{3pt}}p{.74in}| 
@{\hspace{3pt}}p{.7in}| 
@{\hspace{3pt}}p{.6in}| 
@{\hspace{3pt}}p{.7in}| 
@{\hspace{3pt}}p{.6in}|
} 
\hline  ~~ & SPOT & SPOT+SN & SCOT & BSPT$'$ & BSCT$'$ & BST
\\ 
\hline\hline $\in$-theory & ZF & ZF & ZF+ADC & ZF & ZF+ADC & ZFC
\\
\hline Transfer & yes & yes & yes & yes & yes & yes
\\ 
\hline Idealization & countable & countable & countable & standard
params & standard params & full
\\ 
\hline Standardiz. & SP & SP;  \mbox{standard} params & SC;
\mbox{standard} params & SP & SC & full
\\ 
\hline Countable st-$\in$ Choice & no & no & yes & no & yes & Standard -size
 Choice 
\\ \hline
\end{tabular}
\]
}
\drawing
\caption{\textsf{Theories and their properties}}
\label{t6111}
\end{table}

\bigskip
\section{Theory $\SPOT$ and Calculus with infinitesimals} 
\label{starzf}

\subsection{Some consequences of $\SPOT$}
The axioms of $\SPOT$ were given in Section~\ref{s13}.

\begin{lemma}\label{downclosed} The theory $\SPOT$ proves the following:
$$\forall^{\st} n \in \bbN \; \forall m \in \bbN \;(m < n \imp \st(m) ).$$
\end{lemma}

\begin{proof}
Given a standard $n \in \bbN$ and $m < n$, let $A =\{k \in \bbN \,\mid\, k < m\}$.
By $\SP'$ there is a standard $B \subseteq \bbN$ such that for all standard  $k$, $k \in B $ iff $k \in A$ iff $ k < m$. The set $B \subseteq \bbN$ is bounded above by $n$ (Transfer), so it has a greatest element $k_0$ ($<$ is a well-ordering of $\bbN$) , which is standard by Transfer. Now we have $k_0 < m$ and $k_0 + 1 \nless m$, so $k_0 + 1 = m $ and $m $ is standard.
\end{proof}

\begin{lemma} (Countable Idealization)\label{countideal} 
Let $\phi$ be an $\in$-formula with arbitrary parameters.
The theory\, $\SPOT$ proves the following:
$$\forall^{\st}  n \in \bbN\;\exists x\; \forall m \in \bbN\; (m \le n \;\imp  \phi(m,x))\eqi 
\exists x \; \forall^{\st} n \in \bbN \; \phi(n,x).$$
\end{lemma}

\begin{proof}
If $ \forall^{\st} n \in \bbN \; \phi(n,x)$, then, for every standard $ n \in \bbN$, $ \forall m \in \bbN\; (m \le n \imp \phi(m,x) )$, by Lemma~\ref{downclosed}.

Conversely, assume   $\forall^{\st}  n \in \bbN\;\exists x\; \forall m \in \bbN\; (m \le n \;\imp  \phi(m,x))$.
By the Axiom of Separation of $\ZF$, there is a set 
\[
S = \{ n \in \bbN \mid \exists x \; \forall m \in \bbN \; (m \le n \imp \phi(m,x))\},
\]
and the assumption implies that $\forall^{\st} n \in \bbN\; (n \in S)$.

Assume that $S$ contains standard integers only. Then $\bbN \setminus S \neq \emptyset$ by the axiom $\N$. Let $\nu$ be the least element of $\bbN \setminus S$. Then $\nu$ is nonstandard but $\nu - 1$ is standard, a contradiction.

Let $\mu$ be some nonstandard element of $S$. We have $\exists x\;\forall m \in \bbN\; ( m \le \mu \imp \phi(m,x))$; as $n \le \mu$ holds for all standard $n \in \bbN$, we obtain $\exists x\; \forall^{\st} n \in \bbN \; \phi(n,x)$.
\end{proof}

Countable Idealization easily implies the following more familiar form. We use  $\forall^{\st \fin}$ and $\exists^{\st \fin}$ as quantifiers
over standard finite sets.

\begin{corollary} 
Let $\phi$ be an $\in$-formula with arbitrary parameters.  The
theory~\,$\SPOT$ proves the following: For every standard countable
set~$A$
\[
\forall^{\st \fin} a \subseteq A \, \exists x \, \forall y \in a\;
\phi(x,y) \eqi \exists x\, \forall^{\st} y \in A\; \phi(x,y) .
\]
\end{corollary}

The axiom $\SP'$ is often stated and used in the
form
\begin{equation}
 \forall x \in \bbR \;(x \text{ limited } \imp \exists^{\st} r \in \bbR \;(x \approx r)) \tag{$\SP$}
\end{equation}
where $x$ is \emph{limited} iff $|x| \le n$ for some standard $n \in \bbN$, and $x \approx r$ iff $|x - r| \le 1/n $ for all standard $n \in \bbN$, $n \neq 0$. 
The unique standard real number $r$ is called the \emph{standard part of} $x$ or the \emph{shadow of} $x$; notation $\sh(x)$.

We note that in
the statement of $\SP'$, $\bbN$ can be replaced by any countable
standard set $A$.  

\begin{lemma}\label{spandsp'}
The statements $\SP' $ and $ \SP$ are equivalent (over $\ZF + \N + \T$).
\end{lemma}

\begin{proof}[Proof of Lemma~\ref{spandsp'}]
$\SP' \Rightarrow \SP$:
Assume $x \in \bbR $ is limited by a standard  $n_0 \in \bbN$. 
Let $A = \{ q \in \bbQ \,\mid\, q \le x\}$.
Applying $\SP'$ with  $\bbN$ replaced by $ \bbQ$, we obtain a standard set  $B \subseteq \bbQ$ such that 
$\forall^{\st} q \in \bbQ\;(q \in B \eqi q\in A)$.
As $\forall^{\st} q \in B \, (q \le n_0)$ holds, the set $B$ is bounded above (apply Transfer to the formula 
$q \in B \imp q \le n_0$) and so it has a supremum $r \in \bbR$, which is standard (Transfer again).
We claim that $x \approx r$.
If not, then $| x - r | > \frac{1}{n}$ for some standard $n$, hence either $ x < r - \frac{1}{n}$ or $x > r +\frac{1}{n}$. 
In the first case $\sup B \le r - \frac{1}{n}$ and in the second, $\sup B \ge r +\frac{1}{n}$; either way contradicts $\sup B = r$.

$\SP \Rightarrow \SP'$:
The obvious idea is to 
represent the characteristic function of a set $A \subseteq \bbN$
by the binary expansion of a real number  in $[0, 1]$.
But some real numbers have two binary expansions and therefore correspond to two distinct subsets of $\bbN$. This is a source of technical complications that we avoid by using decimal expansions instead.

Given $A \subseteq \bbN$, let $\chi_A$ be the characteristic function of $A$.
Define a real number $x_A = \Sigma_{n = 0}^{\infty} \frac{\chi(n)}{10^{n+1}}$;
as $0 \le x_A \le \frac{1}{9}$, there is a standard real number $r \approx x_A$.
Let $r = \Sigma_{n = 0}^{\infty} \frac{a_n}{10^{n+1}}$ be the decimal expansion of $r$
where for every $n$ there is $k > n$ such that $a_k \neq 9$. Note that if $n$ is standard, then there is a standard $k$ with this property, by Transfer.
If $\chi(n) = a_n$ for all $n$, then $A$ is standard and we let $B = A$. Otherwise 
let $n_0$ be the least $n$ where $\chi(n) \neq a_n$.
From $r \approx x_A$ it follows easily that $n_0$ is nonstandard.
In particular, $a_n \in \{0,1\}$ holds for all standard $n$, hence, by Transfer, for all $n \in \bbN$.
Let $B = \{n \in \bbN \,\mid\, a_n = 1\}$. Then $B$ is standard and for all standard $n \in \bbN$,
$n \in B $ iff $ a_n = 1$ iff $ \chi(n) = 1 $ iff $ n \in A$.
\end{proof}
As explained in the Introduction, Standardization over uncountable sets such as $\bbR$, even for very simple formulas, implies the existence of nonprincipal  ultrafilters over $\bbN$, and so it cannot be proved in $\SPOT$ (consider a standard set $U$ such that $\forall^{\st} X \,(X \in U \eqi X \subseteq \bbN \,\wedge\, \nu \in X)$, where $\nu$ is a nonstandard integer).
But we need to be able to prove the existence of various subsets of $\bbR$ and functions from $\bbR$ to $\bbR$ that arise in the Calculus and may be defined in terms of infinitesimals. 
Unlike the undesirable example above, such uses generally involve Standardization for formulas with standard parameters.

An $\st$-$\in$-formula  $\Phi(v_1, \ldots,v_n)$ is $\deltast$ if it is of the form 
$$Q_1^{\st} \,x_1\ldots  Q_m^{\st} \,x_m\, \psi(x_1,\ldots ,x_m,v_1, \ldots,v_n)$$
 where $\psi $ is an $\in$-formula and $Q$ stands for $\exists$ or $\forall$. 

\begin{lemma}\label{stinternal}
Let $\Phi (v_1,\ldots,v_n)$ be a $\deltast$  formula
with standard parameters. 
Then $\SPOT $ proves: $\quad$
$ \forall^{\st} S\;\exists^{\st} P\; \forall^{\st} v_1,\ldots,v_n\;$
$$ \bigl(\la v_1,\ldots,v_n \ra \in P \eqi
\la v_1,\ldots,v_n \ra \in S \; \wedge \; \Phi (v_1,\ldots,v_n)\bigr) .$$

\end{lemma}

\begin{proof}
Let $\Phi (v_1\ldots, v_n)$ be $Q_1^{\st} \,x_1\ldots  Q_m^{\st} \,x_m\, \psi(x_1,\ldots ,x_m,v_1, \ldots,v_n)$
and  $\phi (v_1\ldots, v_n)$ be
 $Q_1 \, x_1\ldots  Q_m \,x_m\, \psi(x_1,\ldots ,x_m,v_1, \ldots,v_n)$.
By Transfer, 
$ \Phi (v_1\ldots,v_n) \eqi \phi (v_1\ldots,v_n)$ for all standard $ v_1\ldots,v_n$.
The  set $P = \{ \la v_1,\ldots,v_n \ra  \in S \, \mid \, \phi( v_1,\ldots,v_n)\}$ exists by the Separation Principle of $\ZF$, and has the required property.
\end{proof}

 This result has twofold importance:
\begin{itemize}
\item
The meaning of every predicate that for standard inputs is defined by a   $Q_1^{\st} \,x_1\ldots  Q_m^{\st} \,x_m\, \psi$  formula with standard parameters is automatically extended to all inputs, where it it given by the  $\in$-formula $Q_1 \, x_1\ldots  Q_m \,x_m\, \psi$.
\item
Standardization holds for all $\in$-formulas  with additional predicate symbols, as long as all these  additional predicates are defined by $\deltast$  formulas with standard parameters.
\end{itemize}

In $\BST$ all $\st$-$\in$-formulas are equivalent to
$\deltast$ formulas (see Kanovei and Reeken~\cite{KR},
Theorem 3.2.3). In $\SPOT$ the equivalence is true only for certain
classes of formulas, but they include  definitions of all the basic
concepts of the Calculus and much beyond.

We recall that $h \in \bbR$ is \emph{infinitesimal} iff $0 <|h| < \frac{1}{n}$ holds for all standard $n \in \bbN$, $n >  0$.
We use $\forall^{\insl}$ and $\exists^{\insl}$ for quantifiers ranging over infinitesimals and $0$.
The basic concepts of the Calculus have infinitesimal definitions that involve a single alternation of such quantifiers.
The following proposition strengthens a result in Vop\v{e}nka~\cite{V}, p.\;148. 
It shows that the usual infinitesimal definitions of Calculus concepts are $\deltast$.
The variables $x, y $ range over $\bbR$ and $m,n, \ell$ range over $\bbN \setminus \{0\}$. 

\begin{proposition}\label{vopenka}  In $\SPOT$ the following is true:
Let $\phi(x,y)$ be an $\in$-formula with arbitrary parameters.  Then \quad
$\forall^{\insl} h \;  \exists^{\insl} k\;  \phi(h, k) \eqi$
\[
\forall^{\st} m\;\exists^{\st} n\;\forall x\; [\, |x| < 1/n \imp \exists y\; ( |y| < 1/m \;\wedge \;\phi(x,y))\,].
\]
\end{proposition}

By duality, we also have: \quad
$\exists^{\insl} h\; \forall^{\insl} k \;\phi(h, k) \eqi$ $$  \exists^{\st} m \;
\forall^{\st} n\;\exists x\; [\,  |x| < 1/n \;\wedge \;\forall y\; (  |y| <1/m\imp \phi(x, y))\,].$$ 

\begin{proof}
The formula $\forall^{\insl} h \;  \exists^{\insl} k\; \phi(h, k) $ means:
$$  \forall x\; [\, \forall^{\st} n\;  ( |x| < 1/n) \imp \exists y\;\forall^{\st} m\; ( |y| < 1/m \;\wedge \;\phi(x,y))\,],$$ 
where we assume that the variables $m, n$ do not occur freely in $\phi(x,y)$. 
Using Countable Idealization (Lemma~\ref{countideal}), we rewrite this as 
\[
\forall x\; [\, \forall^{\st} n\;  ( |x| < 1/n) \imp \forall^{\st} m\; \exists y\; \forall \ell \le m \,( |y| < 1/\ell \;\wedge \;\phi(x,y))\,].
\]
We now use the observation that $\forall \ell \le m \,( |y| < 1/\ell)$
is equivalent to $ |y| < 1/m$, and the rules $(\alpha \imp
\forall^{\st} v\, \beta) \eqi \forall^{\st} v\, (\alpha \imp \beta) $
and $(\forall^{\st} v\, \beta \imp \alpha )\eqi \exists^{\st} v\,
(\beta \imp \alpha) $, valid assuming that $v$ is not free in $\alpha$
(note that Transfer implies $\exists n\, \st(n)$).  This enables us to
rewrite the preceding formula as follows:
\[
\forall x\; \forall^{\st} m\; \exists^{\st} n\; [\, |x| < 1/n \imp
\exists y\; ( |y| < 1/m \;\wedge \;\phi(x,y))\,].
\]
After exchanging the order of the first two universal quantifiers, we
obtain the formula
$$ \forall^{\st} m\;  \forall x\;  \exists^{\st} n\;  [\, |x| < 1/n \imp \exists y\;  |y| < 1/m \;\wedge \;\phi(x,y))\,],$$ 
to which we apply (the dual form of) Countable Idealization to get 
\[
\forall^{\st} m\; \exists^{\st} n\; \forall x\; \exists \ell \le n\,
[\, |x| < 1/\ell \imp \exists y\; ( |y| < 1/m \;\wedge
\;\phi(x,y))\,].
\]
After rewriting $ \exists \ell \le n\, [\, |x| < 1/\ell \imp \ldots]$
as [\,$ \forall \ell \le n\, (\, |x| < 1/\ell) \imp \ldots]$ and
replacing $\forall \ell \le n\, ( |x| < 1/\ell) $ by $ |x| < 1/n$, we
obtain
$$ \forall^{\st} m\; \exists^{\st} n\;  \forall x\;   [\,  |x| < 1/n \imp \exists y\; (  |y| < 1/m \,\wedge\,\phi(x,y))\,],$$
proving the proposition.
\end{proof}

\subsection{Mathematics in $\SPOT$}

We give some examples to illustrate how infinitesimal analysis works
in $\SPOT$.

\begin{example}
If $F$ is a standard real-valued function on an open interval $(a,b)$
in $\bbR$ and $a, b, c, d$ are standard real numbers with $c \in (a,
b)$, we can define
\begin{equation}
 F'(c) = d \eqi  \forall^{\insl} h\; \exists^{\insl} k \;
\left( \;h \neq 0 \imp  \frac{F( c + h) - F(c)}{h} =  d + k\right).  
\end{equation}
Let $\Phi (F, c, d)$ be the formula on the right side of the
equivalence in~$(1)$.  Lemma~\ref{stinternal} establishes that the formula $\Phi$ is equivalent to a $\deltast$ formula, and $\phi(F, c, d)$ provided by the proof of Lemma~\ref{stinternal}
is easily seen to be equivalent to the standard $\varepsilon$-$\delta
$ definition of derivative.  For any standard $F$, the set $F' = \{
\la c, d \ra \mid \phi(F, c, d) \}$ is standard; it is the derivative
function of $F$.
\end{example}

Proposition~\ref{vopenka} generalizes straightforwardly to all  formulas  that have the form 
$\mathsf{A} h_1 \ldots \mathsf{A} h_n \;  \mathsf{E} k_1 \ldots \mathsf{E} k_m\; \phi(h_1,\ldots, k_1,\ldots \, ,v_1,\ldots)$
 or \\
$\mathsf{E}h_1 \ldots \mathsf{E} h_n \;  \mathsf{A} k_1 \ldots \mathsf{A} k_m\;\phi(h_1,\ldots, k_1,\ldots \, ,v_1,\ldots)$
 where each $\mathsf{A}$ is either $\forall$ or $\forall^{\insl}$, and each $\mathsf{E}$ is either $\exists$ or $\exists^{\insl}$. All such formulas are equivalent to  $\deltast$ formulas.

Formulas of the form $\mathsf{Q}_1 h_1\, \ldots \,\mathsf{Q}_n h_n \,\phi(h_1,\ldots,h_n, v_1,\ldots,v_k)$
where each $\mathsf{Q}$ is either $\forall$ or $\forall^{\insl}$ or $\exists$ or $\exists^{\insl}$, but all quantifiers over infinitesimals are of the same kind (all existential or all universal), are also  $\deltast$. As an example, 
$\exists^{\insl} h \,\forall y\, \exists^{\insl} k\, \phi(h,k,x,y)$  is equivalent to 
$$\exists h \,\forall y\, \exists  k\, \forall^{\st}m\,\forall^{\st} n\, ( |h| < 1/m \,\wedge\, |k| < 1/n \,\wedge\, \phi(h,k,x,y)).$$
The two quantifiers over standard elements of $\bbN$ can be replaced by a single one:
$$\exists h \,\forall y\, \exists  k\, \forall^{\st}m\, ( |h| < 1/m \,\wedge\, |k| < 1/m \,\wedge\, \phi(h,k,x,y)),$$
and then  moved to the front using Countable Idealization.

Klein and Fraenkel proposed two benchmarks for a useful theory of infinitesimals (see Kanovei et al.~\cite{KKKM}):
\begin{itemize}
\item
a proof of  the Mean Value Theorem by infinitesimal techniques;
\item
a definition of the definite integral in terms of infinitesimals.
\end{itemize}
The theory $\SPOT$ easily meets these criteria. The usual nonstandard proof of the Mean Value Theorem
(Robinson~\cite{R}, Keisler~\cite{K1, K2}) uses Standard Part and Transfer, and is easily carried out in $\SPOT$. 
The familiar infinitesimal definition of the Riemann integral for standard bounded functions on a standard interval $[a, b]$ also makes sense in $\SPOT$ and can be expressed by a  $\deltast$ formula. In the next example we outline a treatment inspired by Keisler's use of hyperfinite Riemann sums in~\cite{K2}.

\begin{example}\textbf{Riemann Integral.} \label{riemann}\\
We fix a positive infinitesimal $h$ and the corresponding ``hyperfinite time line'' 
$\bbT = \{ t_i \,\mid\, i \in \bbZ\}$ where $t_i = i\cdot h$. Let $f$ be a standard real-valued function continuous on the standard interval $[a, b ]$. Let $i_a, i_b$ be such that $ i_a \cdot h- h < a \le  i_a \cdot h$ and 
$i_b \cdot h < b \le i_b\cdot h + h$. Then
\begin{equation}
\int_a^b f(t)\, dt = \sh \left( \Sigma_{i = i_a}^{ i_{b} } f(t_i)\cdot h  \right).
\end{equation} 
It is easy to show that the value of the integral does not depend on the choice of $h$. 
We thus have, for standard $f, a,b, r: \quad \int_a^b f(t)\, dt  = r $ \; iff 
$$
\forall^{\insl} h \,\exists^{\insl} k \,\left(\Sigma_{i = i_a}^{  i_{b}  } f(t_i)\cdot h  = r + k\right) \text{ iff }
\exists^{\insl} h \,\exists^{\insl} k \, \left(\Sigma_{i = i_a}^{ i_{b}  } f(t_i)\cdot h  = r + k\right).
$$
The formulas are of the form $\mathsf{A}\,\mathsf{E}$ and $\mathsf{E}\,\mathsf{E}$ respectively, and therefore 
equivalent to  $\deltast$ formulas.

The approach generalizes easily to the Riemann integral of bounded functions on $[a, b]$.
We say that $\cT_h = \la t'_i\ra_{i = i_a}^{i_{b} } $
is an $h$- \emph{tagging on} $[a, b]$ if $i\cdot h \le t'_i \le (i+1)\cdot h$ for all $i = i_a,\ldots, i_b -1$
and $i_b\cdot h \le t'_{i_b} \le b$.
Then for standard $f,a,b,r$
\begin{itemize}
\item
$f$ is Riemann integrable on $[a, b]$ and  
$  \int_a^b f(x)\, dx  = r $ iff
\item
$\forall^{\insl} h\, \forall \cT_h \,\exists^{\insl} k \,\left(\Sigma_{i = i_a}^{  i_{b} } f(t'_i)\cdot h  = r + k\right)$ iff
\item
$\exists^{\insl} h\, \forall \cT_h \,\exists^{\insl} k \,\left(\Sigma_{i = i_a}^{  i_{b} } f(t'_i)\cdot h  = r + k\right).$
\end{itemize}
These formulas are 
again equivalent to  $\deltast$ formulas (the first one is of the form 
$\mathsf{A}\,\mathsf{A}\,\mathsf{E}$ and in the second one both quantifiers over standard sets are existential).
\end{example}

The tools available in $\SPOT$ enable nonstandard definitions and proofs in parts of mathematics that go well beyond the Calculus.
\begin{example}\textbf{Fr\'{e}chet Derivative.}
Given standard normed vector spaces  $V$ and $W$,   a standard open subset $U$ of $V$,  
a standard function $f : U \to W$,  a standard bounded linear operator $A : V \to W$ and a standard $x \in U$;
$A$ is the \emph{Fr\'{e}chet derivative of $f$ at }$x \in U$ iff
$$ \forall z \, \forall^{\insl} h \,\exists^{\insl} k \, \left( \parallel z \parallel_V = h > 0 \imp
\frac{\parallel f(x + z) - f(x) - A \cdot z \parallel_W}{ \parallel z \parallel_V} = k\right).
$$
This definition is equivalent to a $\deltast$ formula.
\end{example}

In Section~\ref{forcingexternal} we show that Standardization for arbitrary formulas with standard parameters can be added to $\SPOT$ and the resulting theory is still conservative over $\ZF$.
This result enables one to dispose of any concerns about the form of the defining formula.


\section{Theory $\SCOT$ and Lebesgue measure}
\label{scot}

We recall (see Section~\ref{s13}) that $\SCOT$ is $\SPOT + \ADC + \SN +\CC$,
where the principle $\CC$ of Countable $\st$-$\in$-Choice postulates
the following.

\medskip
$\CC\;$ Let $\phi (u,v)$ be an $\st$-$\in$-formula with arbitrary
parameters. Then $\forall^{\st} n \in \bbN\; \exists x\; \phi(n,x)
\imp \exists f\, (f \text{ is a function} \,\wedge\, \forall^{\st} n
\in \bbN\, \phi(n, f(n)).$\\

The set $\bbN$ can be replaced by any standard countable set $A$.  We
consider also the principle $\SC$ of Countable Standardization.

\medskip
$\SC$ (Countable Standardization) Let $\psi(v)$ be an
$\st$-$\in$-formula with arbitrary parameters. Then
\[
\exists^{\st} S\; \forall^{\st} n \; (n \in S \eqi n \in \bbN
\,\wedge\, \psi(n)).
\]

\begin{lemma}
\label{s31}
The theory $\SPOT+\CC$ proves $\SC$.
\end{lemma}

\begin{proof}
 Let $\phi(n,x)$ be the formula $`` (\psi(n) \,\wedge\, x = 0 )\, \vee \,  (\neg \psi(n) \,\wedge\,x = 1 )$''.
If $f$ is a function provided by $\CC$, let $A = \{ n \in \bbN \,\mid\, f(n) = 0\}$. By $\SP$ there is a standard set $S$  such that, for all standard $n \in \bbN$, $n \in S$ iff $ n \in A$ iff $ \psi (n) $ holds.
\end{proof}

We introduce an additional principle $\CC^{\st}$.

\medskip
$\CC^{\st}$\; Let $\phi (u,v)$ be an $\st$-$\in$-formula with
arbitrary parameters.  Then
 $$ \forall^{\st} n \in \bbN\; \exists^{\st} x \,  \phi(n,x) \imp \exists^{\st}   F\, (F \text{ is a function}
\,\wedge\, \forall^{\st} n \in \bbN\, \phi(n, F(n)).$$

The principle $\CC_{\bbR}^{\st}$ is obtained from $\CC^{\st}$ by
restricting the range of the variable $x$ to $\bbR$.
\begin{lemma}
The theory $\SPOT + \CC$ proves  $\CC_{\bbR}^{\st}$.
\end{lemma}

\begin{proof}
First use the principle $\CC$ to obtain a function $f : \bbN \to \bbR$
such that $\forall^{\st} n \in \bbN \,(f(n) \in \bbR \,\wedge\,
\st(f(n)) \,\wedge \,\phi(n, f(n)).$ Next define a relation $r
\subseteq \bbN \times \bbN$ by $\la n, m \ra\in r $ iff $ m \in f(n)$.
By Lemma~\ref{s31}, $\SC$ holds.  By $\SC$ there is a standard
$R\subseteq \bbN \times \bbN$ such that $\la n, m \ra\in R $ iff $\la
n, m \ra\in r$ holds for all standard $\la n, m \ra$.  Now define $F:
\bbN \to \bbR$ by $F(n) = \{ m \, \mid\, \la n, m \ra\in R\} $.  The
function $F$ is standard and, for every standard $n$, the sets $F(n) $
and $f(n)$ have the same standard elements. As they are both standard,
it follows by Transfer that $F(n) = f(n)$.
\end{proof}

The full principle $\CC^{\st}$  can conservatively be added to $\SCOT$; see Proposition~\ref{ccst}.

A useful consequence of $\SC$ is the ability to carry out external induction.

\begin{lemma}\label{extind} (External Induction)
Let $\phi(v)$ be an $\st$-$\in$-formula with arbitrary parameters. Then $\SPOT + \SC$ proves the following:
$$[\,\phi(0) \,\wedge\, \forall^{\st} n \in \bbN\, (\phi(n) \imp \phi(n+1))\imp \forall^{\st} n \,\phi(n)\,].  $$
\end{lemma}

\begin{proof}
$\SC$ yields a standard set $S \subseteq \bbN$ such that $\forall^{\st} n \in \bbN\, (n \in S \eqi \phi(n)).$
We have $0 \in S$ and $\forall^{\st} n \in \bbN \,(n \in S \imp n+1 \in S)$. Then
$\forall n \in \bbN \,(n \in S \imp n+1 \in S)$ by Transfer, and  $S = \bbN$ by induction.
Hence $\forall^{\st} n \in \bbN \, \phi(n)$ holds.
\end{proof}

In Example~\ref{lebesgue} it is convenient to use the language of
external collections. Let $\phi (v)$ be an $\st$-$\in$-formula with arbitrary
parameters. We use dashed curly braces to denote the \emph{external
collection} $\pmbaa x \in A \,\mid\, \phi (x) \pmbbb $.  We emphasize that this
is merely a matter of convenience; writing $z \in \pmbaa x \in A
\,\mid\, \phi (x) \pmbbb $ is just another notation for $\phi (z)$.

Standardization in $\BST$ implies the existence of a standard set $S$ such that 
$\forall^{\st} z\; (z \in S \eqi z \in \pmbaa x \in A \,\mid\, \phi (x) \pmbbb )$.
We do not have Standardization over uncountable sets in $\SCOT$, but one important case can be proved.

\begin{lemma}\label{extinf}
Let $\phi (v)$ be an $\st$-$\in$-formula with arbitrary parameters.
Then $\SCOT$ proves that $\mathbf{inf}^{\,\st}\, \pmbaa r \in \bbR \,\mid\,
\phi(r) \pmbbb$ exists.
\end{lemma}

The notation indicates the greatest standard $s \in \bbR$ such that $s \le r$ for all standard $r$ with the property $\phi(r)$ ($+\infty$ if there is no such $r$).

\begin{proof}
Consider $\mathbf{S} = \pmbaa q \in \bbQ \,\mid\,\exists^{\st} r \in
 \bbR\, ( q \ge r \,\wedge\,\phi(r) ) \pmbbb$.  As $\bbQ$ is countable, the principle
 $\SC$ implies that there is a standard set $S$ such that $\forall q
 \in \bbQ\, (q \in S \eqi q \in \mathbf{S})$.  Therefore $\inf S$
 exists ($+\infty$ if $S = \emptyset$) and it is what is meant above
 by $\textbf{inf}^{\;\st}\, \pmbaa r \in \bbR \,\mid\, \phi(r) \pmbbb$.
\end{proof}

We give two examples of mathematics in $\SCOT$.

\begin{example}\textbf{Peano Existence Theorem.}
Peano's Theorem asserts that every first-order differential equation of the form $y' = f(x,y)$ has a solution (not necessarily a unique one) satisfying the initial condition $y(0)=0$, under the assumption that $f$ is continuous in a neighborhood of $\la 0,0 \ra$. The infinitesimal proof begins by constructing the sequences
$$x_0=0, \; x_{k+1} = x_k + h \text{ where } h> 0 \text{ is infinitesimal};$$
$$ y_0 = 0, \; y_{k+1} = y_k + h\cdot f(x_k, y_k).$$
One then shows that there is $N\in \bbN$ such that $x_k, y_k$ are defined for all  $k \le N$, $a = \st(x_N) > 0$, and for some standard $M>0$, $| y_k | \le M \cdot a$ holds for all $k \le N$. 
The desired solution is a standard function $Y: [0, a] \to \bbR$ such that for all standard $x \in [0,a]$, 
if $x \approx x_k$, then $Y(x) \approx y_k$.  On the face of it one needs Standardization over $\bbR$ to obtain this function, but in fact $\SC$ suffices.
Consider the countable set $A = (\bbQ \times \bbQ) \cap ([0,a] \times [-M \cdot a, M \cdot a])$.
By $\SC$, there is a standard $Z \subseteq A$ such that for all standard $\la x,y \ra$, $\la x,y \ra \in Z$ iff 
$\exists k \le N\;(x \approx x_k \,\wedge\, (y \approx y_k \,\vee\, y \ge y_k)$.
Define a standard function $Y_0$ on $\bbQ \cap [0, a]$ by  $Y_0 (x) = \inf \{ y \,\mid\, \la x, y \ra \in Z\}$.
 It is easy to verify that $Y_0$ is continuous on $\bbQ \,\cap\, [0,a]$ and that its extension $Y$ to a continuous function on $[0,a]$ is the desired solution.\\
\end{example}

\begin{example}\textbf{Lebesgue measure.}\label{lebesgue}
In a seminal  paper~\cite{Loeb}  Loeb introduced measures on the external power set of ${}^{\ast}\bbR$ which became known as Loeb measures, and used them to construct the Lebesgue measure on $\bbR$. Substantial use of external collections is outside the scope of this paper (see Subsection~\ref{externalsets}), but it is possible to eliminate the intermediate step and give an  infinitesimal definition \`{a} la Loeb of the Lebesgue measure in internal set theory. We outline here how to construct the Lebesgue outer measure on $\bbR$ in $\SCOT$.

Let $\bbT$ be a hyperfinite time line (see Example~\ref{riemann}) and
let $E \subseteq \bbR$ be  standard. A finite set $A \subseteq \bbT$ \emph{covers}  $E$
if $$\forall t \in \bbT\, (\exists^{\st} x \in E\;( t \approx x)  \imp t \in A).$$
We define $\pmb{\mu}$ by setting
\begin{equation}  \pmb{\mu}(E) = \textbf{inf}^{\,\st}\, \pmbaa  r \in \bbR \,\mid\, r \approx |A| \cdot h \text{ for some  } A \text{ that covers } E  \pmbbb.
\end{equation}
The collection whose infimum needs to be taken is external, but the existence of the infimum is justified by Lemma~\ref{extinf}.
It is easy to see that the value of $\pmb{\mu}(E)$ is independent of the choice of the infinitesimal $h$ in the definition of $\bbT$.
Thus the  external function $\pmb{\mu}$  can be  defined for standard $E\subseteq \cP (\bbR)$ by an $\st$-$\in$-formula with no parameters (preface the formula on the right side of (3)  by $\forall^{\insl} h$ or $\exists^{\insl} h$). 
The principle $\SN$ 
(see Subsection~\ref{s13} and Section~\ref{forcingexternal})
yields a standard  function $m$ on $\cP (\bbR)$ such that $m(E) = \pmb{\mu}(E)$ for all standard $E\subseteq  \bbR$. We prove that $m$ is $\sigma$-subadditive.

Let $E = \bigcup_{n=0}^{\infty} E_n$ where $E$ and the sequence $\la E_n \,\mid\, n \in \bbN\ra$ are standard. 
If $\Sigma_{n=0}^{\infty} m(E_n) = +\infty$ the claim is trivial, so we assume that $m(E_n) = r_n \in \bbR$ for all $n$. Fix a standard  $\varepsilon > 0$. For every standard $n \in \bbN$ there exists $A$ such that  $\phi(n, A)$: $``A \text{ covers } E_n \,\wedge\, |A| \cdot h < r_n +\varepsilon/ 2^{n+1}$'' holds.
By Countable $\st$-$\in$-Choice there is a sequence $\la A_n \,\mid\, n \in \bbN \ra$ such that for all standard $n$ 
 $\phi(n, A_n)$ holds. By Countable Idealization (``Overspill'') there is a nonstandard  $\nu \in \bbN$ such that 
 $   |A_n| \cdot h < r_n +\varepsilon/ 2^{n+1} $ holds for all $n \le \nu$.
We  let $A = \bigcup_{n=0}^{\nu} A_n$.
Clearly $A$ is finite and  covers $E$. Thus for $r =  \sh(|A| \cdot h)$ we obtain
$m(E) \le r$ and 
$$|A| \cdot h \le \Sigma_{n=0}^{\nu} |A_n|\cdot h < \Sigma_{n=0}^{\nu} r_n + \varepsilon.$$
Since the sequence $\Sigma_{n=0}^{\infty} r_n$ converges, we have 
$\sh(\Sigma_{n=0}^{\nu} r_n) = \Sigma_{n=0}^{\infty} r_n $ and $m(A) \le \Sigma_{n=0}^{\infty} r_n +\varepsilon$. As this is true for all standard $\varepsilon > 0$, we conclude that 
$m(E) \le \Sigma_{n=0}^{\infty} r_n = \Sigma_{n=0}^{\infty} m(E_n)$.
\qed

\bigskip
For closed intervals $[a,b]$, $m([a,b]) = b-a$:
Compactness of $[a, b]$ implies that $\forall t \in \bbT\cap [a,b]\, \exists^{\st} x \in E\;( t \approx x) $.
Thus if $A$ covers $[a, b]$ then $A \supseteq \bbT\cap [a,b]$; and for $A = \bbT\cap [a,b]$ one sees easily that $|A| \cdot h \approx (b-a)$. 
With more work, one can show that $m(E)$ coincides with  the conventionally defined Lebesgue  outer measure of $E$ for all standard  $E \subseteq \bbR$.
See Hrbacek~\cite{H2} Section 3
 for more details and other equivalent nonstandard definitions of the Lebesgue outer  measure.%
\footnote{In~\cite{H2} Remark (3) on page 22 it is erroneously claimed that the statement $m_1(A) = r$ is equivalent to an internal formula. The existence of the function $m_1$ there follows from Standardization, just as in the case of $m$ above. }
One can define Lebesgue measurable sets from $m$ in the usual way. One can also define Lebesgue \emph{inner} measure for standard $E$  by
\begin{align*}\mu^{-}(E) = 
 &\textbf{sup}^{\st}\, \pmbaa   r \in \bbR \,\mid\, r \approx |A| \cdot h \text{ 
for some  } A \text{ such that} \\
&\forall t \in \bbT\, (t \in A \imp \exists^{\st} x \in E\;( t \approx x) ) \pmbbb
\end{align*}
and prove that a standard bounded $E \subseteq \bbR$ is Lebesgue measurable iff $m(E) = m^{-}(E)$, and the common value is the Lebesgue measure of $E$; see Hrbacek~\cite{H3}.
\end{example}


\section{Conservativity of $\SPOT$ over $\ZF$}

In this section we apply forcing techniques to prove conservativity of $\SPOT$ over $\ZF$.

\begin{theorem}\label{TheoremA}
The theory $\SPOT$ is a conservative extension of $\ZF$:\\
If $\theta$ is an $\in$-sentence, then $(\,\SPOT \vdash \theta\,)$ implies that $(\,\ZF \vdash \theta\,)$.
\end{theorem}

Theorem~\ref{TheoremA} = Theorem \textbf{A} is an immediate consequence of the following proposition.

\begin{proposition}\label{countablemodels}
Every countable model $\cM = ( M, \in^{\cM} ) $ of $\ZF$ has a countable extension $\scM =(\sM, \in^{\ast}, \st) $ to a model of $\SPOT$ in which  $M$ is the class of all standard sets.
\end{proposition}

 \emph{Proof of Theorem ~\ref{TheoremA}}.
Suppose $\SPOT \vdash \theta$ but $\ZF \nvdash \theta$, where $\theta$ is an $\in$-sentence.
Then the theory $\ZF + \neg \theta$ is consistent, therefore it has a countable model $\cM$, 
by G\"{o}del's Completeness Theorem. Using Proposition~\ref{countablemodels} one obtains its extension  $\scM \vDash \SPOT$, so in particular $\scM \vDash \theta$ and, by Transfer in 
 $\scM$, $\cM \vDash \theta$. This is a contradiction.
\qed

The rest of this section is devoted to the proof of Proposition~\ref{countablemodels}.

\subsection{Forcing according to Enayat and Spector}\label{SS1}
We combine the forcing notion used by Enayat~\cite{E}  to construct end extensions of models of arithmetic, with the one used by Spector in ~\cite{Spr}  to produce extended ultrapowers of models $\cM$ of $\ZF$ by an ultrafilter $\cU \in \cM$.

In this subsection we work in $\ZF$, define our forcing notion and prove its basic properties. 
The next subsection deals with generic extensions of countable models of $\ZF$ and the resulting extended ultrapowers. The general reference to forcing and generic models in set theory is Jech~\cite{J2}.

The set of all natural numbers is denoted $\bbN$ and letters $m, n, k,\ell$ are reserved for variables ranging over $\bbN$. The index set over which the ultrapowers will eventually be constructed is denoted $I$.  In this section we assume $I = \bbN$. A subset $p$ of $\bbN$ is called \emph{unbounded} if $\forall m\, \exists n \in p\, (n \ge m)$ and \emph{bounded} if it is not unbounded.
Of course unbounded is the same as infinite, and bounded is the same as finite. We use this terminology with a view to  Section~\ref{idealization}, where the construction is generalized to $I = \cP^{\fin}(A)$ for any infinite set $A$. The notation $\forall^{\infy} i \in p$ (\emph{for almost all} $i \in p$)  means $\forall i \in p \setminus c$ for some bounded $c$.

As usual, the symbol $\bbV$ denotes the universe of all sets, and $V_{\alpha}$ ($\alpha$ ranges over ordinals) are the ranks of the von Neumann cumulative hierarchy.
We let $\bbF$ be the class of all functions with domain $I$.
The notation $\emptyset_k$ stands for the $k$-tuple $\la \emptyset,\ldots, \emptyset \ra$.

\begin{definition}
Let $\bbP = \{ p \subseteq I \mid p \text{ is unbounded}\}$. 
For $p, p' \in \bbP$ we say that $p'$ \emph{extends} $p$ (notation: $p' \le p$) iff $p' \subseteq p$.

Let $\bbQ = \{  q \in \bbF \mid\, 
\exists k \in \bbN \; \forall i \in I \, (q(i) \subseteq \bbV^{k} \,\wedge\, q(i) \neq \emptyset )\}$. 
The number $k$ is the \emph{rank} of $q$. We note that $q(i)$ for each $i \in I$, and $q$ itself, are sets, but $\bbQ$ is a proper class.
We let $\bar{1} = q$ where $q(i) =\{ \emptyset \}$ for all $i \in I$; $\bar{1}$ is the only $q \in \bbQ$ of rank $0$.

The forcing notion $\bbH$ is defined as follows: $\bbH = \bbP \times \bbQ$ and $\la p', q'\ra \in \bbH$ \emph{extends} $\la  p, q\ra \in \bbH$ (notation: $\la p', q'\ra  \le \la  p, q\ra$) iff $p'$ extends $p$, $\rank q' = k' \ge k = \rank q$,
and for almost all $i \in p'$ and all $\la x_0,\ldots, x_{k'-1} \ra \in q'(i)$,  $\la x_0,\ldots,x_{k-1}\ra \in q(i)$.
Every $\la p, q \ra \in \bbH$ extends $\la p, \bar{1}\ra$.
\end{definition}

The poset $\bbP$ is used to force a generic filter over $I$ as in Enayat~\cite{E}, and 
$\bbH$ forces an extended ultrapower of $\bbV$ by the generic filter $\cU$ forced by $\bbP$. It is a modification of the forcing notion from Spector~\cite{Spr}, with the difference that in~\cite{Spr} $\cU$ is not forced but assumed to be a given ultrafilter in $\bbV$.

A set $D \subseteq \bbP$ is \emph{dense in} $\bbP$ if for every $p \in \bbP$ there is $p' \in D$
such that $p'$ extends $p$.
We note that for any set $S \subseteq I$, the set 
$D_S  = \{ p \in \bbP \,\mid\, p \subseteq S \,\vee\, p \subseteq I \setminus S \} $ is dense in $\bbP$.

Similarly, a class $E \subseteq \bbH$ is \emph{dense in} $\bbH$ if for every $\la p, q \ra  \in \bbH$ there is 
$\la p' , q' \ra\in E$ such that $\la p' , q' \ra \le \la p, q \ra$.

\medskip
The \emph{forcing language} $\fL$ has a constant symbol $\check{z}$ for every  $z \in \bbV$ (which we identify with $z$ when no confusion threatens), and a constant symbol $\dot{G}_n$ for each  $n \in \bbN$.
Given an $\in$-formula $\phi( w_1,\ldots, w_r, v_1,\ldots,v_s)$, we define the \emph{forcing relation} $\la p, q \ra  \Vdash \phi(\check{z}_1,\ldots,\check{z}_r, \dot{G}_{n_1},\ldots,\dot{G}_{n_s})$ for $\la p, q \ra \in \bbH$ by meta-induction on the logical complexity of $\phi$.
We use $\neg, \wedge$ and $\exists$ as primitives and consider the other logical connectives and quantifiers as defined in terms of these. Usually, we suppress the explicit listing  in $\phi$ of the constant symbols $\check{z}$
for the elements of $\bbV$.

\begin{definition}\label{forcing}(Forcing relation.)
\begin{enumerate}
\item
$\la p, q \ra  \Vdash \check{z}_{1} = \check{z}_{2}$  iff
$z_1 = z_2$.
\item
$\la p, q \ra  \Vdash \check{z}_{1} \in \check{z}_{2}$  iff
$z_1 \in z_2$.
\item
$\la p, q \ra  \Vdash \dot{G}_{n_1} = \dot{G}_{n_2}$  iff 
$\rank q = k > n_1, n_2$ and \\
$\forall^{\infy} i \in p \, \forall \la x_0,\ldots, x_{k-1}\ra \in q(i)\;( x_{n_1}= x_{n_2}).$
\item
$\la p, q \ra  \Vdash \dot{G}_{n_1} \in \dot{G}_{n_2}$  iff 
$\rank q = k > n_1, n_2$ and \\
$\forall^{\infy} i \in p \, \forall \la x_0,\ldots, x_{k-1}\ra \in q(i)\; (x_{n_1}\in x_{n_2}).$
\item
$\la p, q \ra  \Vdash \dot{G}_{n} = \check{z}$  iff $\la p, q \ra  \Vdash \check{z} = \dot{G}_{n}$ iff
$\rank q = k > n$ and \\
$\forall^{\infy} i\in p \, \forall \la x_0,\ldots, x_{k-1}\ra \in q(i)\;( x_{n}= z).$
\item
$\la p, q \ra  \Vdash \check{z} \in \dot{G}_{n}$ iff 
$\rank q = k > n$ and \\
$\forall^{\infy} i\in p \, \forall \la x_0,\ldots, x_{k-1}\ra \in q(i)\; (z \in x_{n}).$
\item
$\la p, q \ra  \Vdash \dot{G}_{n} \in \check{z}$ iff
$\rank q = k > n$ and \\
$\forall^{\infy} i\in p \, \forall \la x_0,\ldots, x_{k-1}\ra \in q(i)\; (x_{n}\in z).$
\item
$\la p, q \ra  \Vdash \neg \phi( \dot{G}_{n_1},\ldots,\dot{G}_{n_s})$ iff  
$\rank  q = k > n_1,\ldots, n_s$ and 
there is no $\la p', q' \ra $ extending $\la p, q \ra$  such that 
$\la p', q' \ra \Vdash \phi( \dot{G}_{n_1},\ldots,\dot{G}_{n_s}).$
\item
$\la p, q \ra  \Vdash (\phi \,\wedge\, \psi)( \dot{G}_{n_1},\ldots,\dot{G}_{n_s})$ iff \\
$\la p, q \ra  \Vdash \phi( \dot{G}_{n_1},\ldots,\dot{G}_{n_s})$ and $\la p, q \ra  \Vdash \psi( \dot{G}_{n_1},\ldots,\dot{G}_{n_s})$.
\item
$\la p, q \ra  \Vdash \exists v\, \psi( \dot{G}_{n_1},\ldots,\dot{G}_{n_s}, v)$ iff $\rank  q = k > n_1,\ldots, n_s$ and 
for every $\la p', q' \ra $ extending $\la p, q \ra$ there exist $\la p'', q'' \ra $ extending $\la p', q' \ra$ and $m \in \bbN$ such that 
$\la p'', q'' \ra \Vdash \psi(\dot{G}_{n_1},\ldots,\dot{G}_{n_s}, \dot{G}_{m}).$
\end{enumerate}
\end{definition}

\begin{lemma}\label{basic1}$\,$(Basic properties of forcing)
\begin{enumerate}
\item
If $ \la p, q \ra \Vdash \phi$ and $ \la p', q' \ra$ extends $ \la p, q \ra$, then 
$ \la p', q' \ra \Vdash \phi$.
\item
No $ \la p, q \ra $ forces both $\phi$ and $ \neg \phi$.
\item
Every $\la p, q \ra$ extends to $ \la p', q' \ra$ such that 
$\la p', q' \ra \Vdash \phi$ or\\ $ \la p', q' \ra \Vdash \neg \phi$. 
\item
If $\la p, q \ra \Vdash \phi$ and $p' \setminus p$ is bounded, then $\la p', q \ra \Vdash \phi$.
\end{enumerate}
\end{lemma}

\begin{proof}
(1) - (3) are immediate from the definition of forcing and (4) can be proved by induction on the complexity of $\phi$.
\end{proof}
The following proposition establishes a relationship between this forcing and ultrapowers.
\begin{proposition}\label{Los} (``\L o\'{s}'s  Theorem'')
Let $\phi(v_1,\ldots,v_s)$ be an $\in$-formula with parameters from $\bbV$.\\
Then
$\la p, q \ra  \Vdash  \phi(\dot{G}_{n_1},\ldots,\dot{G}_{n_s})$ iff
$\rank q  = k > n_1,\ldots, n_s$ and \\
$\forall^{\infy} i \in p \, \forall \la x_0,\ldots, x_{k-1}\ra \in q(i)\; 
\phi(x_{n_1},\ldots, x_{n_s}).$
\end{proposition}

\begin{proof}
For atomic formulas (cases (1) - (7)) the claim is immediate from the definition.
Case (9) is also trivial (union of two bounded sets is bounded).

Case (8): Let 
$c = \{i \in p\,\mid\, \forall \la x_0,\ldots,x_{k-1}\ra \in q (i) \;\neg \phi( x_{n_1},\ldots,x_{n_s})\}$.\\
We need to prove that $\la p, q \ra  \Vdash \neg \phi$, iff $p \setminus c$ is bounded.

Assume that  $\la p, q \ra  \Vdash \neg \phi$ and $p \setminus c $ is unbounded. We let $p' = p \setminus c$ and 
$q' (i) =\{ \la x_0,\ldots,x_{k-1}\ra \in q (i) \,\mid\, \phi(x_{n_1},\ldots,x_{n_s}) \}$ for $i \in p'$,
$q'(i) = \{\emptyset_k\}$ for $i \in I \setminus p'$. Then $\la p', q' \ra \in \bbH$ extends $\la p, q \ra$ and, by the inductive assumption, $\la p', q' \ra \Vdash \phi$, a contradiction. 

Conversely, assume  $\la p, q \ra  \nVdash \neg \phi$ and $p \setminus c $ is bounded.
Then there is $\la p', q' \ra $ of rank $k'$ extending $\la p, q \ra$ such that $\la p', q' \ra \Vdash \phi$. By the inductive assumption, there is a bounded set $d$ such that  
$$\forall i \in (p' \setminus d) \, \forall \la x_0,\ldots,x_{k'-1}\ra \in q' (i) \; \phi( x_{n_1},\ldots,x_{n_s}).$$
But $(p \setminus c) \cup d$ is a bounded set, so there exist $i \in (p' \cap c)\setminus d$. For such $i$ and $\la x_0,\ldots,x_{k'-1}\ra \in q' (i)$ one has both $\neg \phi(x_{n_1},\ldots ,x_{n_s})$
and $ \phi( x_{n_1},\ldots, x_{n_s})$, a contradiction.

Case (10):\\
Let  
$c = \{i \in p\,\mid\, \forall \la x_0,\ldots,x_{k-1}\ra \in q (i) \;\exists v\; \psi( x_{n_1},\ldots,x_{n_s}, v)\}.$
We need to prove that $\la p, q \ra  \Vdash \exists v\,  \psi$ iff $p \setminus c$ is bounded.

Assume that  $\la p, q \ra  \Vdash \exists v\, \psi$ and $p \setminus c $ is unbounded. 
We let $p' = p \setminus c $ and 
$q' (i) =\{ \la x_0,\ldots,x_{k-1}\ra \in q (i) \,\mid\, \neg\, \exists v \, \psi(x_{n_1},\ldots,x_{n_s}, v) \}$ for $i \in p'$; $q'(i) = \{\emptyset_k\}$ for $i \in I \setminus p'$.
Then $\la p', q'\ra$ extends $\la p, q \ra$ and,
by the definition of $\Vdash$, there exist  $\la p'', q'' \ra $  extending $\la p', q' \ra $ with $ \rank q''= k''$, and $m < k''$ such that 
$\la p'', q'' \ra  \Vdash \psi(\dot{G}_{n_1},\ldots,\dot{G}_{n_s},\dot{G}_m)$.
By the inductive assumption, there is a bounded set $d$ such that 
$$\forall i \in (p'' \setminus d)\, \forall \la x_0,\ldots,x_{k''-1}\ra \in q'' (i) \; \psi(x_{n_1},\ldots x_{n_s}, x_m).$$ Hence 
$$\forall i \in (p'' \setminus d)\, \forall \la x_0,\ldots,x_{k''-1}\ra \in q'' (i) \; \exists v\, \psi(x_{n_1},\ldots x_{n_s}, v).$$
But $i \in p'' \setminus d$ implies $i \in p'$; this contradicts the definition of $q'$.

Assume that  $p \setminus c$ is bounded. 
By the Reflection Principle in $\ZF$ there is a least von Neumann rank $V_{\alpha} $ such that for all $i \in c$ and all $\la x_0,\ldots, x_{k-1}\ra \in q(i)$ there exists $v \in V_{\alpha}$ such that $ \psi (x_{n_1},\ldots x_{n_s}, v)$.
Let $\la p', q' \ra$ be any condition extending $\la p, q \ra$ and let $k' =\rank  q' $. 
We let  $p'' =  p' \cap c$ and    
\begin{align*}  
q''(i) = &\{ \la x_0,\ldots, x_{k'-1}, x_{k'}\ra  \,\mid\\
& \la x_0,\ldots, x_{k'-1}\ra \in q'(i) \,\wedge \,
\psi (x_{n_1},\ldots x_{n_s}, x_{k'}) \,\wedge\, x_{k'} \in V_{\alpha}\} 
\end{align*}
for $i \in p''$, $q''(i) =\{\emptyset_{k'}\}$ otherwise. Then $\la p'', q'' \ra $ extends $\la p', q' \ra$ and, by the inductive assumption, 
$\la p'', q'' \ra \Vdash \psi  (\dot{G}_{n_1},\ldots,\dot{G}_{n_s}, \dot{G}_{k'})$.
This proves that $\la p, q \ra  \Vdash \exists v\, \psi$.
\end{proof}

We observe that if $q$ is in $\bbQ$ and $\ell < k =\rank q$, then $q\uh \ell$ defined by 
$(q\uh \ell) (i) = \{ \la x_0, \ldots, x_{\ell-1} \ra  \,\mid\, \exists x_{\ell},\ldots, x_{k-1} \,
\la x_0, \ldots, x_{k-1} \ra \in q(i) \}$ is in $\bbQ$.

\begin{corollary}
If  $\rank  q = k > n_1,\ldots, n_s$, $\la p', q' \ra$ extends $\la p, q \ra$ and 
$\la p', q' \ra  \Vdash  \phi(\dot{G}_{n_1},\ldots,\dot{G}_{n_s})$, then 
$\la p', q' \uh k \ra$ extends $\la p, q \ra$ and 
$\la p', q'\uh k \ra  \Vdash  \phi(\dot{G}_{n_1},\ldots,\dot{G}_{n_s})$.
\end{corollary}

\begin{lemma}\label{fixf}
Let $z \in \bbV$. 
For every $\la p, q \ra$ there exist $\la p, q' \ra$ extending $\la p, q \ra$ and  $m < k' = \rank  q'  $ such that $\la p, q' \ra \Vdash \check{z} = \dot{G}_m$.
\end{lemma}

\begin{proof}
Let $q'(i) = \{ \la x_0,\ldots, x_{k-1}, x_{k}\ra  \,\mid\, \la x_0,\ldots, x_{k-1}\ra \in q(i) 
 \,\wedge\, x_k =z\}$, and $m = k$,  $k' = k+1$.
\end{proof}

We write $\la p, q\ra \Vdash \check{z} = \dot{G}_m$ as $\la p, q\ra \Vdash z = \dot{G}_m$, and $\la p, q\ra \Vdash \check{z} \in \dot{G}_m$ as $\la p, q\ra \Vdash z\in \dot{G}_m$.
We say that $\la p, q \ra$ \emph{decides} $\phi$ if $\la p, q \ra \Vdash \phi$ or $ \la p, q \ra \Vdash \neg \phi$. 
The following lemma is needed for the proof that the extended ultrapower satisfies $\SP$.

\begin{lemma}\label{decideN}
For every $\la p, q\ra $ and $m <  k =\rank  q$ there is 
$\la p', q'\ra  $ that extends $\la p, q\ra $ and is such that for every $n \in \bbN$,
$\la p', q'\ra $ decides $ n \in \dot{G}_m$.
\end{lemma}

\begin{proof}
We first construct a sequence $\la \la p_n, q_n \ra\,\mid\, n \in \bbN \ra$ 
such that $\la p_0, q_0 \ra = \la p, q \ra $ and for each $n$, $\rank q_n = k$, 
 $p_{n+1} \subset p_n$, $\la p_{n+1}, q_{n+1}\ra$ extends $\la p_{n }, q_{n }\ra$, and 
$\la p_{n+1}, q_{n+1}\ra $ decides $n\in \dot{G}_m$.

Given $p_n$,
let $c= \{i \in p_n\,\mid\, \forall \la x_0,\ldots,x_{k-1}\ra \in q_n(i) \, (n \in x_m)\}$. 
If $c$ is unbounded, we let $p'_{n+1} = c$ and $q_{n+1} = q_n$.
Otherwise $p_n \setminus c$ is unbounded and 
 we let $p'_{n+1} = p_n \setminus c$ and $q_{n+1}(i) =\{ \la x_0,\ldots,x_{k-1}\ra \in q_n(i) \;\mid\, n \notin x_m\}$ for $i \in p'_{n+1}$,  $q_{n+1}(i) =\{\emptyset_k\}$ otherwise.
We obtain $p_{n+1} $ from $p'_{n+1} $ by omitting the least element of $p'_{n+1} $.
Proposition~\ref{Los} implies that $\la p_{n+1}, q_{n+1}\ra$ decides $n\in \dot{G}_m$.

Let $i_n$ be the least element of $p_{n}$.
We define $\la p' , q'\ra$ as follows:
$i \in p'$ iff $i \in  p_n$ and $q'(i) = q_n(i)$, where $i_n \le i < i_{n+1}$;
$q'(i) =\{\emptyset_k\}$ otherwise.
It is clear from the construction and Proposition~\ref{Los} that $\la p', q'\ra \in \bbH$, it extends $\la p, q\ra$, and it decides  $n \in \dot{G}_m$ for every $n \in \bbN$.
\end{proof}


\subsection{Extended ultrapowers}
\label{extult}
In this subsection we define the extended ultrapower of a countable model of $\ZF$ by a generic filter $\cU$, prove some fundamental properties of this structure, and conclude that it is a model of $\SPOT$. 

We take Zermelo-Fraenkel set theory as our metatheory, but the proof employs  very little of its powerful machinery.  Subsection~\ref{finitistic} explains how the proof given below can be converted into a finitistic proof.

We use $\omega$ for the set of natural numbers in the metatheory, and  $r,s$ as variables ranging over $\omega$.
A set $S$ is \emph{countable} if there is a mapping of $\omega$ onto $S$.

Let $\cM = (M, \in^{\cM})$ be a countable model of $\ZF$. Concepts defined in Subsection~\ref{SS1} make sense in $\cM$ and all results of~\ref{SS1} hold in $\cM$. When $\cM$ is understood, we use the notation and terminology from~\ref{SS1} for the  concepts in the sense of  $\cM$; thus $\bbN$ for $\bbN^{\cM}$, ``unbounded'' for ``unbounded in the sense of $\cM$'', $\bbP$ for $\bbP^{\cM}$, $\Vdash$ for 
``$\Vdash$ in the sense of $\cM$'', etc. The model $\cM$ need not be well-founded externally, and $\omega$ is isomorphic to an initial segment of $\bbN$ which may be proper.

\begin{definition}
$\cU \subseteq M$ is a \emph{filter on} $\bbP $ if 
\begin{enumerate}
\item
$\cM  \vDash``p \in \bbP$'' for every $p \in \cU$;
\item
If $p \in \cU$ and $\cM \vDash ``p' \in \bbP \,\wedge\,  p \text{ extends } p'$'', then $p' \in \cU$;
\item
For eny $p_1, p_2 \in \cU$ there is $p \in \cU$ such that $\cM \vDash ``p \text{ extends } p_1 \,\wedge\, p \text{ extends } p_2$''.
\end{enumerate}

A filter $\cU$ on $\bbP$ is $\cM$-\emph{generic} if for every $D\in M$ such that $\cM \vDash ``D\text{ is dense in } \bbP$'' there is $p \in \cU$ for which $p \in^{\cM} D$.
\end{definition}

Since $\bbP$ has only countably many dense subsets  in  $\cM$, $\cM$-generic filters are easily constructed by recursion. Let $\la p_{r} \,\mid\,r \in \omega \ra$ be an enumeration of $\bbP$ and $\la D_{r} \,\mid\,r \in \omega \ra$ be an enumeration of all dense subsets of $\bbP$ in $\cM$. Let $q_0 =p_0$ and for each $s \in \omega $ let $q_{s+1} = p_r$ for the least $r$ such that  $\cM \vDash ``p_{r} \text{ extends } q_{s} \,\wedge\,p_{r} \in \cD_{s}$''.Then let $\cU = \{ p \in M \,\mid \, \cM \vDash ``p \in \bbP \,\wedge\, q_{s} \text{ extends } p \text
{'' for some } s \in \omega\}$. 

$\cM$-generic filters $\cG \subseteq M \times M$ on $\bbH$ are defined and constructed analogously.

\begin{lemma}
If $\cG$ is an $\cM$-generic filter on $\bbH$, then $\cU = \{ p \in \bbP \,\mid\, \exists q \, \la p, q\ra \in \cG\}$ is an $\cM$-generic filter on $\bbP$.
\end{lemma}

\begin{proof}
In $\cM$: if $D$ is dense in $\bbP$, then $\{ \la p, q\ra  \,\mid\, p \in D \,\wedge\, q \in \bbQ\}$ is dense in $\bbH$.
\end{proof}

We now define the \emph{extended ultrapower of $\cM$ by $\cU$}; we follow closely the presentation in Spector~\cite{Spr}.

Let $\Omega = \{ m \in M \,\mid\, \cM  \vDash ``m \in \bbN \text{''}\}$. We define binary relations $=^{\ast}$ and $\in^{\ast}$ on $\Omega$ as follows:\\
$m =^{\ast} n$ iff there exists $\la p, q\ra \in \cG $ such that $ \rank  q = k > m,n $ and 
$\la p, q\ra \Vdash \dot{G}_m =\dot{G}_n$;\\
$m \in^{\ast} n$ iff there exists $\la p, q\ra \in \cG $ such that $ \rank  q = k > m,n $ and 
$\la p, q\ra \Vdash \dot{G}_m \in\dot{G}_n$.

It is easily seen from the definition of forcing and Proposition~\ref{Los} that $ =^{\ast} $ is an equivalence relation on $\Omega$, and a congruence relation with respect to $\in^{\ast}$. We denote the equivalence class of $m \in \Omega$ in the  relation  $ =^{\ast} $ by $G_m$, define $G_m \in^{\ast} G_n$ iff $m \in^{\ast} n$,
and let $N = \{ G_m \,\mid\, m \in \Omega\}$.
The \emph{extended ultrapower of $\cM$ by $\cU$} is the structure $\fN = (N, \in^{\ast})$.

There is a natural embedding $j$ of $\cM $ into $\fN$ defined as follows:  By Lemma~\ref{fixf} 
for every $z \in M$ there exist $\la p, q\ra \in \cG $ and $ m < \rank q  $ such that 
$\la p, q\ra  \Vdash z =\dot{G}_m$.
We let $j(z) = G_m$ and often identify $j(z)$ with~$z$.
It is easy to see  that the definition is independent of the choice of representative from  $G_m$, and that $j$ is an embedding of $\cM$  into $\fN$.

\begin{proposition}\label{fteu}
(The Fundamental Theorem of Extended Ultrapowers)
Let $ \phi(v_1,\ldots,v_s )$ be an $\in$-formula with parameters from $M$. \\
If  $G_{n_1},\ldots,G_{n_s} \in N$, then  the following statements are equivalent:
\begin{enumerate}
\item
$\fN \vDash \phi(G_{n_1},\ldots,G_{n_s} )$.
\item
There is some $\la p, q \ra \in \cG$ such that $\la p, q \ra  \Vdash \phi( \dot{G}_{n_1},\ldots,\dot{G}_{n_s} )$
holds in $\cM$.
\item 
There exists some $\la p, q \ra \in \cG$ with $\rank q = k > n_1,\ldots, n_s$ such that  
$\cM \vDash \forall i \in p \, \forall \la x_0,\ldots, x_{k-1}\ra \in q(i)\; \phi( x_{n_1},\ldots, x_{n_s}).$
\end{enumerate}
\end{proposition}

\begin{proof}
Statement (3) is just a reformulation of  (2) using Proposition~\ref{Los} plus the fact that if $d$ is bounded, then $\la p, q \ra \in \cG$ implies $\la (p \setminus d), q\ra \in \cG$. (Observe that $D =\{ \la p', q' \ra \,\mid\,  \la p', q' \ra \le  \la p, q \ra
\,\wedge\, p' \le (p \setminus d)\}$ is dense in $\la p, q\ra$; for the definition of ``dense in'' see the sentence preceding Lemma~\ref{forall}.)

The equivalence of  (1) and (2) is the Forcing Theorem. It is proved as usual,  by induction on the logical complexity of $\phi$.
The cases $v_1 = v_2$ and $v_1 \in v_2$ follow immediately from the definitions of $=^{\ast}$ and $\in^{\ast}$, and the conjunction is immediate from (3) in the definition of a filter. 

We consider next the case where $\phi$ is of the form $\neg \psi$.
First assume that $\fN \vDash \neg\, \psi (G_{n_1},\ldots,G_{n_s} )$.
Lemma~\ref{basic1}~(3), implies that there exists $\la p, q \ra \in \cG$ such that either $\la p, q \ra  \Vdash  \psi(\dot{G}_{n_1},\ldots,\dot{G}_{n_s}) $ or $\la p, q \ra  \Vdash \neg  \psi(\dot{G}_{n_1},\ldots,\dot{G}_{n_s})$.
In the first case $\fN \vDash \psi(G_{n_1},\ldots,G_{n_s} )$ by the  inductive assumption; a contradiction.
Hence $\la p, q \ra  \Vdash \neg  \psi(\dot{G}_{n_1},\ldots,\dot{G}_{n_s})$.

Assume that 
$\fN \nvDash \neg\, \psi (G_{n_1},\ldots,G_{n_s} )$; then $\fN \vDash  \psi(G_{n_1},\ldots,G_{n_s} )$ and the inductive assumption yields
$\la p', q' \ra \in \cG$ such that $\la p', q' \ra  \Vdash  \psi(G_{n_1},\ldots,\dot{G}_{n_s}) $.
There can thus be no $\la p, q \ra \in \cG$ such that $\la p, q \ra  \Vdash  \neg \psi(\dot{G}_{n_1},\ldots,\dot{G}_{n_s}) $.

Finally, we assume that $\phi(u_1,\ldots,v_s)$ is of the form $\exists w \, \psi(u_1,\ldots,v_s, w)$.
If $\fN \vDash \phi(G_{n_1},\ldots,G_{n_s})$ then  $\fN \vDash \psi(G_{n_1},\ldots,G_{n_s}, G_m )$
for some $m \in \Omega$. By the inductive assumption there is some $\la p, q \ra \in \cG $ such that $\la p, q \ra \Vdash \psi(\dot{G}_{n_1},\ldots,\dot{G}_{n_s}, \dot{G}_m )$ and hence, by the definition of forcing,  $\la p, q \ra  \Vdash \exists w\,\psi(\dot{G}_{n_1},\ldots,\dot{G}_{n_s},w)$.

Conversely, if $\la p, q \ra \in \cG $ and $\la p, q \ra  \Vdash \exists w\,\psi(\dot{G}_{n_1},\ldots,\dot{G}_{n_s},w )$, then by the definition of forcing there are $\la p', q' \ra \in \cG $ and $m \in \Omega$ such that 
$\la p', q' \ra  \Vdash \psi(\dot{G}_{n_1},\ldots,\dot{G}_{n_s}, \dot{G}_m )$. By the inductive assumption, 
$\fN \vDash \psi(G_{n_1},\ldots,G_{n_s}, G_m )$ holds, and hence $\fN \vDash \phi(G_{n_1},\ldots,G_{n_s} )$.
\end{proof}

\begin{corollary}\label{elememb}
The embedding $j$ is an elementary embedding of $\cM$ into $\fN$.
\end{corollary}

\begin{corollary}\label{elemZF}
The structure $\fN$ satisfies $\ZF$.
\end{corollary}

\begin{proposition}\label{ext}
The structure $\widehat{\fN} = (N, \in^{\ast} , M)$ satisfies the principles of Transfer, Nontriviality and Standard Part.
\end{proposition}

\begin{proof} 
Transfer is Corollary~\ref{elememb}. 

Working in $\cM$, 
for every $\la p, q \ra$ with $\rank q =k$ define $q'$ of rank $k+1$ by
$q'(i) = \{ \la x_0,\ldots, x_{k-1}, x_{k}\ra  \,\mid\, \la x_0,\ldots, x_{k-1}\ra \in q(i) 
 \,\wedge\, x_k = i\}$ and note that $\la p, q' \ra$ extends $\la p, q \ra$.
By $\cM$-genericity of $\cG$ some such $\la p, q'\ra$ belongs to $\cG$.
It is easily seen from the Fundamental Theorem that  $G_k$ is an integer in $\fN $ and that $\fN \vDash ``G_k\neq n$'' for all $n \in \Omega$.
Hence $\mathbf{O}$ holds in  $\widehat{\fN}$.

It remains to prove the Standard Part principle.  
 Let $G_m \in N$. By Lemma~\ref{decideN} there is $\la p, q \ra \in \cG$ which decides  $n \in \dot{G}_m$ for all $n\in \Omega$. 
The set $E =\{ n \in \Omega \,\mid\, \la p, q\ra \Vdash n\in \dot{G}_m\}$ is definable in $\cM$, hence there is $e \in M$  such that  $\cM \vDash ``n \in e $'' iff $ n \in E$ iff $\fN \vDash ``n \in G_m$''. Thus
$ \widehat{\fN}  \vDash ``\st(e) \,\wedge\, \forall^{\st} \nu \in \bbN\,(\nu \in e \eqi \nu \in G_m)$''.
\end{proof}
\emph{This  proves Proposition~\ref{countablemodels}, and hence   Theorem~\textbf{A}.}
 \qed


\section{Conservativity of $\SCOT$ over $\ZFc$}
\label{stand}

In this section we show that if $\ZF$ is replaced by $\ZFc$, the Standard Part principle  can be strengthened to Countable $\st$-$\in$-Choice.

We recall that $\ZFc$ implies $\ACC$; this provides enough choice to prove that the ordinary ultrapower of a countable model $\cM$ of  $\ZFc$ by an $\cM$-generic filter $\cU$ on $\bbP$ satisfies \L o\'{s}'s Theorem and thus yields an elementary extension of $\cM$. 
A proof that every countable model $\cM = (M, \in^{\cM})$ of $\ZFc$ has an extension to a model of $\SPOT + \SC$ in which $M$ is the class of all standard sets can be obtained by a straightforward adaptation of the arguments in Enayat's paper~\cite{E}, in particular, of the proofs of Theorems B and C there. Essentially, all one has to do is to replace countable models of second-order arithmetic by countable models of $\ZFc$. We followed this approach in an early version of the present paper.

Another proof of conservativity of $\SCOT$ over $\ZFc$ was suggested to us by Kanovei in a private communication. 
Its basic idea is to use forcing to add to $\cM$ a mapping of $\omega_1$ onto $\bbR$ without adding any reals. In the resulting generic extension there are nonprincipal  ultrafilters over  $\bbN$, and one can take an ultrapower of $\cM$ by one of them to obtain an extension that satisfies $\SCOT$. 
However, this method does not seem adequate for handling Idealization over uncountable sets in Section~\ref{idealization}.

We first outline the simplification of the forcing that is possible in the presence of $\ACC$, and then use it to prove that, assuming $\cM$ is a countable model of $\ZFc$, the structure $\widehat{\fN}$ from Proposition~\ref{ext} satisfies also $\CC$ and $\SN$. The proof can be viewed as a warm-up for similar but more complex arguments of the following sections.

We work in $\ZFc$. Given $I = \bbN $ and $q \in \bbQ$ of rank $k$, $\ACC$ guarantees the  existence of a function $f$  such that $\forall i \in I \,  (f(i)\in q(i))$; 
let $\widehat{q}(i) = \{f(i)\}$ where $f(i) = \la f_0(i),\ldots, f_{k-1}(i) \ra$.
For any $\la p, q \ra \in \bbH$ the condition $\la p, \widehat{q}\,\ra $ extends $\la p, q \ra$.
We could replace $\bbQ$ by 
$\widetilde{\bbQ} = \{ \widetilde{q} \in \bbF \, \mid\, 
\exists k \, \forall i \in I \, (\widetilde{q}(i) \in \bbV^{k})\}.$ 
But there is  no need at all for the symbols $\dot{G}_k$, $k \in \bbN$, and Spector's component $\bbQ$ of the forcing notion $\bbH$, if our forcing language allows names for all $f \in \bbF$ (see below for details).

On the other hand, Standardization and Countable $\st$-$\in$-Choice, unlike Transfer and Idealization, deal with $\st$-$\in$-formulas, so we  need  to extend our definition of the forcing relation to such formulas. 

The forcing notion we use in this section is $\bbP$.
The forcing language $\widetilde{\fL}$ has a constant symbol $\check{f}$ for every function $f \in \bbF$ (the check is usually suppressed). 
   Forcing is  defined for arbitrary $\st$-$\in$-formulas. 
Only the following clauses in the definiton of the forcing relation are necessary:  

\begin{definition}\label{forcing2}
(Simplified forcing.)

(1') $p  \Vdash f_{1} = f_{2}$  iff
$\forall^{\infy} i \in p \, (f_{1}(i)= f_{2}(i)).$
 
(2') $p  \Vdash f_{1} \in f_{2}$  iff
$\forall^{\infy} i \in p \,( f_{1}(i)\in f_{2}(i)).$
 
(8) $p  \Vdash \neg \phi $ iff  
there is no $p' $ extending $p$  such that 
$p' \Vdash \phi.$

(9) $p \Vdash (\phi \,\wedge\, \psi) $ iff 
$p  \Vdash \phi$ and $p \Vdash \psi $.

(10') $p \Vdash \exists v\, \psi $ iff 
for every $p'$ extending $p$ there exist $p'' $ extending $p'$ and a function $f\in \bbF$ such that 
$p'' \Vdash \psi(f).$

(11) \\$\la p, q \ra  \Vdash \st( f)$ iff 
$\exists x\, \forall^{\infy} i\in p \, (f(i) = x )$ iff
$\forall^{\infy} i, i'\in p \,( f(i) = f(i') )$.
\end{definition}

The basic properties of forcing from Lemma~\ref{basic1} remain valid, but Proposition~\ref{Los} (``\L o\'{s}'s Theorem'') of course holds only for $\in$-formulas, in the form \quad 
$p \Vdash \phi(f_1,\ldots, f_r )\text{ iff }\forall^{\infy} i \in p \;
\phi(f_{1}(i),\ldots, f_{_r}(i)).$

We now take Zermelo-Fraenkel set theory as our metatheory. Let $\cM$ be  a countable model of $\ZFc$, $\cU$ an $\cM$-generic filter on $\bbP \in M$, and $\fN = (N, \in^{\ast})$ the ultrapower of $\cM$ by $\cU$.  If  $\cM \vDash``f \in \bbF$'', then $[f]_{\cU}$ is the equivalence class of $f$ modulo $\cU$.
We identify $x \in M$ with $[c_x]_{\cU}$ where  $c_x$ is the constant function on $I$  with value $x$ in the sense of $\cM$, and let $\widehat{\fN} =  (N, \in^{\ast}, M)$.

Proposition~\ref{fteu} takes the following form.
\begin{proposition}\label{fteu2}
Let $\phi(u_1,\ldots,u_r)$ be an $\st$-$\in$-formula with parameters from  $M$. 
If $\cM \vDash``f_1,\ldots, f_r \in \bbF$'', then  the following statements are equivalent:
\begin{enumerate}
\item
$\widehat{\fN} \vDash \phi([f_1]_{\cU},\ldots,[ f_r]_{\cU})$.
\item
There is some $p \in \cU$ such that $p \Vdash \phi(f_1,\ldots, f_r )$ holds in $\cM$.
\end{enumerate}
\end{proposition}
Corollaries~\ref{elememb} and~\ref{elemZF} and Proposition~\ref{ext} remain valid in this modified setting.
For $\in$-formulas Proposition~\ref{fteu2}  is just a fancy way to state the ordinary \L o\'{s}'s Theorem, but 
for $\st$-$\in$-formulas it provides a useful handle on the behavior of $\widehat{\fN}$. 

We need the following corollary (which can also be proved more tediously  directly from the definition of the forcing relation). 
The statement  ``$D\subseteq \bbP$ is \emph{dense in} $p$'' means that $\forall p'' \le p\, \exists p'\le p''\, (p' \in D)$.

\begin{lemma}\label{forall}
If $p \Vdash \forall^{\st} m \in \bbN \,\phi(m)$, then the set
$D_m = \{ p' \in \bbP \,\mid\, p' \Vdash \phi (m) \}$ is dense in $p$  for every $m \in \bbN$.
\end{lemma}

\begin{proof}
We show that the claim holds in every model $\cM$ of $\ZFc$.
For every $p'' \le p$ in $\cM$ there is an $\cM$-generic filter $\cU$ such that $p'' \in \cU$. 
By Proposition~\ref{fteu2} $\widehat{\fN} \vDash \forall^{\st} m \,\phi(m)$, 
so $\cM \vDash ``m \in \bbN$'' implies 
$\widehat{\fN} \vDash \phi(m)$. By~\ref{fteu2} again, there exists $p' \in \cU$ such that 
$p' \Vdash \phi(m)$. We can take $p' \le p''$.
\end{proof}

\begin{lemma}\label{decideE}
Let $ \phi(u, v)$ be an $\st$-$\in$-formula with parameters from~$\widetilde{\fL}\!.$.
Then $\ZFc$ proves the following: $\quad$
If $p \Vdash \forall^{\st} m \,\exists v\, \phi(m, v)$, then  there is 
$p' \in \bbP $ and a sequence $\la f_m \,\mid\, m \in \bbN \ra$ such that $p'$ extends $p $ and  
$p' \Vdash \phi(m, f_m)$ for every $m \in \bbN$.
\end{lemma}

\begin{proof}
By Lemma~\ref{forall}, the set 
$D_m = \{ p' \in \bbP \,\mid\, p' \Vdash \exists v \,\phi (m, v)\}$ is dense in $p$ for each $m \in \bbN$.
Clause (10') in the definition of simplified forcing implies that also the set 
$E_m = \{ p' \in \bbP \,\mid\,\exists f\in \bbF\; p' \Vdash \phi (m, f) \}$ is dense in $p$.
We let $$\la m', p'\ra \mathbf{R} \la m'',p''\ra\text{ iff } p'' \subset p'\subseteq p \,\wedge\,m'' = m'+1 \,\wedge\,
p' \in E_{m'} \,\wedge \, p'' \in E_{m''}.$$
Applying $\ADC$ to the relation $ \mathbf{R}$ we obtain  a sequence $ \la p_m \,\mid\, m \in \bbN \ra$ 
such that $  p_0  \subseteq p $ and, for each $m$,   
 $p_{m+1} \subset p_m$ and $\exists f \in \bbF\, (p_m \Vdash \phi (m, f))$.
We next use $\ACC$ to obtain a sequence $\la f_m \,\mid\, m \in \bbN \ra$ such that 
$ p_m \Vdash \phi (m, f_m)$.
Note that the Reflection Principle of $\ZF $ provides a set $A = V_{\alpha} $ such that for all $m$, 
$$(\exists f \in \bbF\; p_m \Vdash \phi (m, f) )\imp (\exists f \in\bbF \cap  A \,\; p_m \Vdash \phi (m, f)).$$

As in the proof of  Lemma~\ref{decideN}, let $i_m$ be the least element of $p_{m}$ and 
let $  p' =\bigcup_{m = 0}^{\infty} p_m \cap(i_{m + 1} \setminus i_m)$.
Then  for every $m$ the set $p' \setminus p_m$ is bounded, hence $ p' \Vdash \phi (m, f_m)$.
\end{proof}

\begin{proposition}\label{cc}
If $\cM$ satisfies $\ZFc$, then $\widehat{\fN} $ satisfies $\CC$.
\end{proposition}

\begin{proof}
Assume that $\widehat{\fN}  \vDash \forall^{\st} m \,\exists v\, \phi(m, v)$.
Then there is $p \in \cU$ such that $p \Vdash \forall^{\st} m \,\exists v\, \phi(m, v)$.
By Lemma~\ref{decideE}  there is 
$p' \in \cU$ and a sequence $\la f_m \,\mid\, m \in \bbN \ra$ such that 
$p' \Vdash \phi(m, f_m)$ for every $m \in \bbN$.

We define a function $g$ on $I$ by 
$g(i) = \{ \la m, f_m(i) \ra \,\mid\, m \in \bbN\}$.
Recall that  $\check{g}$ is the name for $g$ in the forcing language. 
By \L o\'{s}'s Theorem, 
$p' \Vdash ``\check{g}$ is a function with domain $\bbN$''\!, and, for every $m \in \bbN$, 
$p' \Vdash \check{g}(m) = f_m$.
We conclude that $\widehat{\fN} \vDash ``\check{g}$ is a function with domain $\bbN$''\!, and $\widehat{\fN} \vDash \forall^{\st} m \in \bbN\,\phi(m, \check{g}(m))$.
\end{proof}

\begin{proposition}\label{ccst}
If $\cM$ satisfies $\ZFc$, then $\widehat{\fN} $ satisfies $\CC^{\st}$.
\end{proposition}

\begin{proof}
The assumption $\widehat{\fN}  \vDash \forall^{\st} m \,\exists^{\st} v\, \phi(m, v)$ implies that 
we can take $f_m = c_{x_m}$ in Lemma~\ref{decideE} and  the proof of Proposition~\ref{cc}.
For the function $g$ on $I$ defined by 
$g(i) = \{ \la m, x_m \ra \,\mid\, m \in \bbN\}$ we then have $\widehat{\fN} \vDash ``\check{g}$ is standard''.
\end{proof}

Let $p, p' \in \bbP$ and let $\gamma$ be an increasing mapping of $p'$ onto $p$; we extend $\gamma$ to $I = \bbN$ by defining $\gamma (a) = 0$ for $a \in I \setminus p$.

\begin{lemma}\label{isomorphism}
$p\Vdash \phi(f_1 , \ldots, f_r)$ iff 
$p'\Vdash \phi(f_1\circ \gamma , \ldots, f_r \circ \gamma) $.
\end{lemma}

\begin{proof}
This follows by induction on the cases in the definition of simplified forcing, using the observation that the mapping $p'' \to   p' \cap \gamma^{-1}[p''] $ is an isomorphism of the
posets $(\{ p'' \in \bbP \,\mid\, p'' \le p\}, \le)$  and $(\{ p'' \in \bbP \,\mid\, p'' \le p'\}, \le)$.
\end{proof}

\begin{corollary}
Let $\phi$ be an $\st$-$\in$-formula with parameters from $\bbV$.
Then $p \Vdash \phi$ iff $p' \Vdash \phi$.
\end{corollary}

\begin{proposition}\label{extsn}
The structure $\widehat{\fN} $ satisfies the principle  of Standardization for $\st$-$\in$-formulas with no parameters.
\end{proposition}

\begin{proof}
For standard $x$, $\widehat{\fN} \vDash \phi(x)$  iff  $p \Vdash \phi(x)$
for every $p \in \bbP$. The right side is expressible by an $\in$-formula.
\end{proof}

\emph{This completes the proof of Theorem \textbf{D}.}\\

Another principle that can be added to $\SCOT$ is Dependent Choice for $\st$-$\in$-formulas.\\

  $\mathbf{DC}\;$  Let $\phi (u,v)$ be an $\st$-$\in$-formula with arbitrary parameters. \\
If $B$ is a set, $b \in B$ and $\forall x \in B\, \exists y \in B\; \phi(x,y)$, then there is a sequence $\la b_n \,\mid\, n \in \bbN\ra$ such that $b_0 = b $ and 
$\forall^{\st} n \in \bbN\, (b_n \in B \,\wedge\, \phi(b_n, b_{n+1}))$.\\

\begin{theorem}
$\SCOT + \mathbf{DC}\;$ is a conservative extension of $\ZFc$.
\end{theorem}

\begin{proof}
We show that $\mathbf{DC}$ holds in the structure $\widehat{\fN} $.

Let $\cM \vDash``b, B \in \bbF$'' and
$p \Vdash b \in B \,\wedge\, \forall x \in B\, \exists y \in B\; \phi(x,y)$. We now let 
$A = \{ \la p', f' \ra \,\mid\, p' \le p \,\wedge\, p' \Vdash f' \in B \}$, note that $\la p, b \ra \in A$,   and define $\mathbf{R}$ on $A$ by
$ \la p', f' \ra \mathbf{R} \la p'', f'' \ra \text{ iff } p'' \le p' \,\wedge\, p'' \Vdash \phi(f', f'').$

It is clear from the properties of forcing that for every $ \la p', f' \ra \in A$ there is $\la p'', f'' \ra  \in A$ such that $ \la p', f' \ra \mathbf{R} \la p'', f'' \ra $. Using $\ADC$ we obtain a sequence 
$\la \la p_n, f_n \ra \,\mid\, n \in \bbN \ra$ such that $p_0 \le p$, $f_0 = b$, and for all $n \in \bbN$ $\la p_n, f_n\ra \in A$, $p_{n+1} \le p_n$, and $p_{n+1} \Vdash \phi(f_n, f_{n+1})$.

The rest of the proof imitates the arguments in the last paragraphs of the proofs of Lemma~\ref{decideE} and Proposition~\ref{cc}.
\end{proof}


\section{Standardization for parameter-free formulas}
\label{forcingexternal}

In this section we prove that the structure $\widehat{\fN} = (N, \in^{\ast} , M)$ satisfies the principle of Standardization for parameter-free formulas assuming only that the model $\cM$ satisfies $\ZF$.
Explicitly, the principle postulates:\\

$\SN$  Let $\phi (v)$ be an $\st$-$\in$-formula with no parameters. Then
$$\forall^{\st} A\, \exists^{\st}  S\; \forall^{\st} x\; (x \in S \eqi x \in A \,\wedge\,\phi(x)). $$

\begin{lemma} 
The principle $\SN$ is equivalent to Standardization for $\st$-$\in$-formulas with standard parameters.
\end{lemma}

\begin{proof}
Given $\phi (v,p_1,\ldots,p_\ell)$ where $p_1,\ldots,p_\ell$ are standard, we let $P= \{\la p_1,\ldots,p_\ell \ra \}$ and apply $\SN$ to the formula $\psi (w)$ with no parameters expressing $``\exists w_1, \ldots, w_{\ell} \, (  w = \la v, w_1, \ldots, w_{\ell} \ra \,\wedge\, \phi(v,w_1,\ldots,w_\ell))$''
and to the standard set $A \times P$.
We get a standard $S$ such that  for all standard inputs
$\la x,y_1,\ldots,y_\ell\ra \in S \eqi  \la x,y_1,\ldots,y_\ell\ra   \in A \times 
P \,\wedge\, \phi (x,y_1,\ldots,y_\ell)$ holds. The set $T = \{x \in A\,\mid\, \la x,p_1,\ldots,p_\ell\ra \in S\}$ standardizes $\phi (v,p_1,\ldots,p_\ell)$. 
\end{proof}

With the exception of the last proposition, in this section we work in $\ZF$. As in Section~\ref{stand}, we begin with extending forcing to $\st$-$\in$-formulas.

We  add the following clauses to Definition~\ref{forcing}:

(11) $\la p, q \ra  \Vdash \st( z)$ for every $z \in \bbV$.

(12) $\la p, q \ra  \Vdash \st( \dot{G}_{n} )$ iff $\rank q = k > n \text{ and } $
$$\exists x\, \forall^{\infy} i\in p \, \forall \la x_0,\ldots, x_{k-1}\ra \in q(i)\; (x_{n}= x) \text{ or, equivalently, } $$
$$\forall^{\infy} i, i'\in p \, \forall \la x_0,\ldots, x_{k-1}\ra \in q(i)\; 
\forall \la x'_0,\ldots, x'_{k-1}\ra \in q(i')\; (x_{n}= x'_{n}) . $$

The basic properties of forcing from Lemma~\ref{basic1} in Subsection~\ref{SS1} remain valid, but  ``\L o\'{s}'s Theorem'' does not hold for $\st$-$\in$-formulas. However,  Proposition~\ref{Los} is instrumental in the proof of Lemma~\ref{decideN}.
The technical lemma that follows is a simple consequence of Proposition~\ref{Los} for forcing of $\in$-formulas, but it remains valid even for forcing of $\st$-$\in$-formulas.

\begin{definition}
For $\sigma = \la n_1,\ldots,n_s \ra$ where $n_1,\ldots,n_s < k$ are mutually distinct
  let $\pi^{k}_{\sigma} : \bbV^k \to \bbV^s$
be the ``projection'' of $\bbV^k$ onto $\bbV^s$:
$$\pi^{k}_{\sigma}(\la x_0,\ldots,x_{k-1}\ra) = \la x_{n_1},\ldots,x_{n_s}\ra.$$
For $q \in \bbQ$ of rank $k$, $q \uh \sigma$ of rank $s$ is defined by 
$(q \uh \sigma)(i) = \pi^k_{\sigma} [q(i)]$.
\end{definition}

\begin{lemma}\label{homogeneity}
Let $\phi$ be an $\st$-$\in$-formula with parameters from $\bbV$.\\
Assume that 
$\rank q_1 =k_1$, $\sigma_1 = \la n_1,\ldots,n_s \ra$ where $n_1,\ldots,n_s <~k_1$, and
$\rank  q_2 =k_2$, $\sigma_2 = \la m_1,\ldots,m_s \ra$ with $m_1,\ldots,m_s <~k_2$.
If \\$(q_1 \uh \sigma_1)(i) = (q_2 \uh \sigma_2)(i)$ for all $i \in p$ (we write 
$q_1 \uh \sigma_1 =_{p} q_2 \uh \sigma_2$),
then
$ \la p, q_1\ra  \Vdash \phi(\dot{G}_{n_1}, \ldots, \dot{G}_{n_s})$ if and only if  
$ \la p, q_2\ra  \Vdash \phi(\dot{G}_{m_1}, \ldots, \dot{G}_{m_s})$.
\end{lemma}

\begin{proof}
The proof is by induction on the logical complexity of $\phi$.

For $\in$-formulas the assertion follows immediately from Proposition~\ref{Los}. In particular, it holds for all atomic formulas involving $\in$ and $=$ (cases (1) - (7) in the definition of forcing). It is also clear for $\st$ (cases (11) and (12)) and for conjunction (case (9)).

Case (8):

Assume that the statement is true for $\phi$, $ \la p, q_1\ra  \Vdash \neg \phi(\dot{G}_{n_1}, \ldots, \dot{G}_{n_s})$, and $ \la p, q_2\ra  \nVdash \neg \phi(\dot{G}_{m_1}, \ldots, \dot{G}_{m_s})$.
Then there exists a condition $ \la p', q'_2\ra\le \la p, q_2\ra$ such that 
$ \la p', q'_2\ra  \Vdash \phi(\dot{G}_{m_1}, \ldots, \dot{G}_{m_s})$.
From the inductive assumption it follows that  $ \la p', q'_2\uh \sigma_2\ra  \Vdash \phi(\dot{G}_{0}, \ldots, \dot{G}_{s-1})$
(recall that $q'_2\uh \sigma_2 \subseteq \bbV^s$).
Let $\bar{q} = q'_2\uh \sigma_2 \le q_2\uh \sigma_2 =_{p} q_1\uh \sigma_1$.
We define $q'_1 =(\pi^{k_1}_{\sigma_1})^{-1} [\bar{q}] \cap q_1\le q_1$.
Now $ \la p', q'_1\ra  \le \la p, q_1\ra $ and $q'_2\uh \sigma_2 =_{p'} q'_1\uh \sigma_1$, so
by the inductive assumption
$ \la p', q'_1\ra  \Vdash \phi(\dot{G}_{n_1}, \ldots, \dot{G}_{n_s})$.
This is a contradiction with $ \la p, q_1\ra  \Vdash \neg \phi(\dot{G}_{n_1}, \ldots, \dot{G}_{n_s})$.
The reverse implication follows by exchanging the roles of  $ \la p, q_1\ra$ and  $ \la p, q_2\ra$.

Case (10):

Assume that $ \la p, q_1\ra  \Vdash \exists v\, \psi(\dot{G}_{n_1}, \ldots, \dot{G}_{n_s},v)$,
Let  $\la p', q'_2\ra \le  \la p, q_2\ra$, where $\rank     q'_2 = k'_2 =  k_2 + r$.
We need to find $\la p'', q''_2\ra $ extending $\la p', q'_2\ra$ and $m$ such that 
$ \la p'', q''_2\ra  \Vdash \psi(\dot{G}_{m_1}, \ldots, \dot{G}_{m_s}, \dot{G}_m)$.
This will prove that $ \la p, q_2\ra  \Vdash \exists v\, \psi(\dot{G}_{m_1}, \ldots, \dot{G}_{m_s},v)$.

We let $\bar{q} = \pi^{k'_2}_{\sigma_2} [q'_2] \subseteq \pi^{k_2}_{\sigma_2} [q_2] =_{p} \pi^{k_1}_{\sigma_1} [q_1] $ and $\bar{\bar{q}} = (\pi^{k_1}_{\sigma_1})^{-1} [\bar{q}] \cap q_1\le q_1$.
We define $q'_1$ of rank $k'_1= k_1 + r$ by\\
$q'_1(i) = 
\{ \la x_0,\ldots, x_{k_1 - 1}, y_0, \ldots, y_{r-1} \ra  \,\mid\,
  \la x_0,\ldots, x_{k_1 - 1}\ra \in \bar{\bar{q}} \;\wedge\,$\\
$\la x_{n_1},\ldots, x_{n_s}\ra      =   
\pi^{k'_2}_{\sigma_2}  (\la x'_0,\ldots, x'_{k_2 - 1}, y_0, \ldots, y_{r-1} \ra ) \text{ for some }$\\
$ \la x'_0,\ldots, x'_{k_2 - 1}, y_0, \ldots, y_{r-1} \ra \in q'_2  \} $\\
for $i \in p';$
 $q'_1(i) = \{\emptyset_{k'_1} \}$ otherwise.
We observe that  $ q'_1(i)  \ne \emptyset$ and $\pi^{k'_1}_{\sigma_1}[q'_1]
=_{p'} \pi^{k'_2}_{\sigma_2}[q'_2]$.

We have $\la p', q'_1\ra \le \la p, q_1\ra$; hence there are $\la p'', q''_1\ra \le \la p', q'_1\ra$ with 
$\rank q''_1 = k''_1$ and $n$
such that $ \la p'', q''_1\ra  \Vdash \psi(\dot{G}_{n_1}, \ldots, \dot{G}_{n_s}, \dot{G}_{n})$.
Finally we construct $q''_2$ of rank $k''_2 = k'_2+1$ such that $\la p'', q''_2\ra \le  \la p'', q'_2\ra $ and, for some $m$,
$\pi^{k''_1}_{\sigma_1\# n}(q''_1) =_{p''} \pi^{k''_2}_{\sigma_2\# m}(q''_2) $, where 
$\sigma_1\# n = \la n_1,\ldots,n_s, n \ra$ and $\sigma_2\# m = \la m_1,\ldots,m_s, m \ra$.
By the inductive assumption, this establishes 
$ \la p'', q''_2\ra  \Vdash  \psi(\dot{G}_{m_1}, \ldots, \dot{G}_{m_s},\dot{G}_{m})$.

We start with $\widehat{q} = \pi^{k''_1}_{\sigma_1}[q''_1] \subseteq \pi^{k'_1}_{\sigma_1}[q'_1]
=_{p'} \pi^{k'_2}_{\sigma_2}[q'_2]$ and define\\
$q''_2(i)  =$
$\{ \la x_0,\ldots, x_{k'_2-1}, z \ra \,\mid$ $ \la x_0,\ldots, x_{k'_2-1} \ra \in 
(\pi^{k'_2}_{\sigma_2})^{-1} [ \widehat{q}\,] \cap q'_2\;\wedge$\\
$ \la x_{n_1},\ldots,x_{n_s}, z\ra \in \pi^{k''_1}_{\sigma_1\# n}(q''_1)  \}$\\
 for $i \in p''$;  $q''_2(i) = \{\emptyset_{k'_2 + 1} \}$ otherwise.\\
We have $\la p'', q''_2\ra \le  \la p'', q'_2\ra $ and $\rank q''_2 = k''_2 = k'_2 + 1$. Let $m = k'_2$. It follows  from the construction that 
$\pi^{k''_1}_{\sigma_1\# n}(q''_1) =_{p''} \pi^{k''_2}_{\sigma_2\# m}(q''_2)$.
\end{proof}

\begin{corollary}
Let $\phi$ be  an $\st$-$\in$-sentence with parameters from $\bbV$.
Then $\la p,q\ra \Vdash \phi$ iff $\la p,\bar{1}\ra \Vdash \phi$.
\end{corollary}

As in Section~\ref{stand}, let $p, p' \in \bbP$ and let $\gamma$ be an increasing mapping of $p'$ onto $p$; we extend $\gamma$ to $I$ by defining $\gamma (a) = 0$ for $a \in I \setminus p$.

\begin{lemma}\label{pullback}
Let $\phi(v_1, \ldots, v_s)$ be an $\st$-$\in$-formula with parameters from $\bbV$. Then
$\la p, q\ra \Vdash \phi(\dot{G}_{n_1}, \ldots, \dot{G}_{n_s})$ iff 
$\la p', q \circ \gamma\ra \Vdash \phi(\dot{G}_{n_1}, \ldots, \dot{G}_{n_s})$.
\end{lemma}

\begin{proof}
As for Lemma~\ref{isomorphism}.
\end{proof}

\begin{corollary}
Let $\phi(v_1,\ldots,v_r)$ be an $\st$-$\in$-formula.  For $z_1,\ldots,z_r \in~\bbV$
$\la p,q\ra \Vdash \phi(z_1,\ldots,z_r)$ iff $\la p', q'\ra \Vdash \phi(z_1,\ldots,z_r)$.
\end{corollary}

\begin{proposition}\label{extsn2}
The structure $\widehat{\fN} = (N, \in^{\ast} , M)$ satisfies the principle  of Standardization for $\st$-$\in$-formulas with no parameters.
\end{proposition}

\begin{proof}
For standard $z$, $\widehat{\fN} \vDash \phi(z)$  iff  $\la p, q\ra \Vdash \phi(z)$
for every $\la p, q\ra \in \bbH$. The right side is expressible by an $\in$-formula.
\end{proof}

\emph{This completes the proof of Theorem \textbf{B}.}

\bigskip
There is yet another principle that can be added to $\SPOT$ and keep it conservative over $\ZF$.
One of its important consequences is the impossibility to uniquely specify an infinitesimal.\\

$\UP$ (Uniqueness Principle) Let $\phi(v)$ be an $\st$-$\in$-formula with standard parameters.
If there exists a unique $x$ such that $\phi(x)$, then this $x$ is standard. 

\begin{theorem}
$\SPOT  + \UP$ is a conservative extension of $\ZF$.
\end{theorem}

\begin{proof}
If $\widehat{\fN} \vDash \exists x\, [\phi(x) \,\wedge\, \forall y\, (\phi(y) \imp y = x)\,\wedge\, \neg \st(x)]$, then there is $\la p, q \ra \in \cG$ and $m < k = \rank q$ such that 
$\la p,q\ra \Vdash \phi(\dot{G}_m) \,\wedge\, \neg \st(\dot{G}_m) \,\wedge\, 
\forall y\, (\phi(y) \imp y = \dot{G}_m)$.

Let $r = q \uh \{ m \}$, i.e., for all $i \in I$, 
$x \in r(i) \eqi \exists \la x_0,\ldots, x_{k-1} \ra \in q(i)\, ( x_m = x)$.

\emph{Claim 1.}
\emph{For every $x \in \bigcup_{i \in p} r(i)$  the set $p_x = \{ i \in p \,\mid\, x \in r(i)\}$ is bounded; we let $i_x$ 
denote its greatest element.}

\emph{Proof of Claim 1.}
Let $x$ be such that $p_x$ is unbounded.  Define
\begin{align*}
q_x(i) &= \{ \la x_0,\ldots, x_{k-1} \ra \in q(i) \,\mid\,( x_k = x ) \}\text{ for }i \in p_x; \\
&= \{\emptyset_k\} \text{ otherwise}.
\end{align*}
Then $\la p_x, q_x \ra \le \la p, q \ra $ and  $\la p_x, q_x \ra \Vdash \dot{G}_m = x$, i.e., 
  $\la p_x, q_x \ra \Vdash \st(\dot{G}_m)$, a contradiction.  \qed

\emph{Claim 2.}
\emph{There exist unbounded mutually disjoint sets $p_1, p_2 \subset p$ and nonempty sets $ s(i) \subseteq r(i)$ for all $i \in p_1 \cup p_2$ such that }\\
$\left( \bigcup_{i \in p_1} s(i) \right)  \cap \left( \bigcup_{i \in p_2} s(i) \right) = \emptyset.$

We postpone the proof of Claim 2 and complete the proof of  the theorem.

We let $\widetilde{q}(i) = \{ \la x_0,\ldots, x_{k-1} \ra \in q(i) \,\mid\, x_k \in s(i) \} $ for $i \in p_1 \cup p_2$,  $\widetilde{q}(i) = \{ \emptyset_k\}$ otherwise.
We have $\la p_1, \widetilde{q} \ra\le \la p, q \ra$, $\la p_2, \widetilde{q} \ra\le \la p, q \ra$, and consequently $\la p_1, \widetilde{q} \ra \Vdash \phi(\dot{G}_m)$, $\la p_2, \widetilde{q} \ra \Vdash \phi(\dot{G}_m)$.

Let $\gamma$ be an increasing mapping of $p_1$ onto $p_2$ extended by $\gamma (i) =0$ for $i \in I \setminus p_1$.
By Lemma~\ref{pullback} $\la p_1, \widetilde{q} \circ \gamma\ra \Vdash \phi(\dot{G}_m)$.
We ``amalgamate''  $\la p_1, \widetilde{q} \ra $ and $\la p_1, \widetilde{q} \circ \gamma \ra $
 to form a condition of rank $2k$ as follows: \qquad
$\la x_0,\ldots,x_{k-1}, y_0,\ldots,y_{k-1} \ra \in \widehat{q}(i) \eqi$
$$\la x_0,\ldots,x_{k-1}\ra \in \widetilde{q}(i)\, \wedge \,
\la y_0,\ldots,y_{k-1}\ra \in \widetilde{q}(\gamma(i)) $$
and observe that $\la p_1, \widehat{q} \ra \le \la p , \widetilde{q} \ra$.
Let $\sigma_1 = k$ and $\sigma_2 = \{ k+\ell \,\mid\, \ell < k\}$. 
We have $\widehat{q} \uh \sigma_1 = \widetilde{q}$ and $\widehat{q} \uh \sigma_2 = \widetilde{q} \circ \gamma$. By Lemma~\ref{homogeneity} $\la p_1, \widehat{q}\ra \Vdash \phi(\dot{G}_m) \,\wedge\, \phi(\dot{G}_{k+m})$. But $x_m \neq y_m$ holds for all $\la x_0,\ldots,x_{k-1}, y_0,\ldots,y_{k-1} \ra \in \widehat{q}(i) $ and all $i \in p_1$, so  $\la p_1, \widehat{q}\ra \Vdash \dot{G}_m \neq \dot{G}_{k +m}$.
This contradicts $\la p_1, \widehat{q}\ra \Vdash \forall y\, (\phi(y) \imp y = \dot{G}_m).$

\emph{Proof of Claim 2.}
W.l.o.g. we can assume $p = I = \bbN$ (map $p$ onto $\bbN$ in an increasing way).
Define sequences $\la n_{\ell} \,\mid\, \ell \in \bbN \ra$, $\la \alpha_{\ell} \,\mid\, \ell \in \bbN \ra$
and $\la s(n_{\ell})\,\mid\, \ell \in \bbN \ra$ by recursion as follows:

Let $n_0 = 0$, $\alpha_0 = \min \{ i_x \,\mid\, x \in r(0)\}$ and  
$s(n_0) = s(0) = \{ x \in r(0) \,\mid\, i_x = \alpha_0 \}$.

At stage $\ell + 1$ let $n_{\ell + 1} = \alpha_{\ell} + 1$, 
$\alpha_{\ell +1} = \min \{ i_x \,\mid\, x \in r(n_{\ell+1} )\}$       and
$s(n_{\ell + 1} ) = \{ x \in r(n_{\ell+1}) \,\mid\, i_x = \alpha_{\ell + 1} \}$.

We observe that $s(n_{\ell}) \cap s(n_{\ell'}) = \emptyset$ for all $\ell \neq \ell'$.
It remains to let $p_1 = \{ n_{2\ell} \,\mid\, \ell \in \bbN\}$ and 
$p_2 = \{ n_{2\ell + 1} \,\mid\, \ell \in \bbN\}$.   \qed
\end{proof}


\section{Idealization}
\label{idealization}

We recall the axioms of the theory $\BST$; see the references Kanovei and Reeken~\cite{KR} and Fletcher et al.~\cite{F} for motivation and more detail.
In addition to the axioms of $\ZFC$, they are:\\

$\B$  (Boundedness)  $\quad \forall x \, \exists^{\st} y \,(x \in y)$.\\

$\T$ (Transfer) 
Let $\phi (v)$ be an $\in$-formula with standard parameters. Then
$$\forall^{\st} x\; \phi(x)  \imp  \forall x\; \phi(x).$$

$\mathbf{S}$ (Standardization)  Let $\phi (v)$ be an $\st$-$\in$-formula with arbitrary parameters. Then
$$\forall^{\st} A\, \exists^{\st}  S\; \forall^{\st} x\; (x \in S \eqi x \in A \,\wedge\,\phi(x)). $$

$\mathbf{BI}$  (Bounded Idealization)
Let $\phi$ be an $\in$-formula with arbitrary parameters. For every set $A$
\[
\forall^{\st\fin} a \subseteq A \;\exists y\; \forall x \in a\;
\phi(x,y) \eqi \exists y \; \forall^{\st} x \in A \; \phi(x,y).
\]

\subsection{Idealization over uncountable sets}
In order to obtain models with Bounded Idealization, the construction of Subsections~\ref{SS1} and~\ref{extult} needs to be generalized from $I =\bbN$ to $I = \cP^{\fin} (A)$, where $A$ is any infinite set. The key is the right definition of  ``unbounded'' subsets of $I$. We work in $\ZF$.

We use the notation $\cP^{\le m} (A)$ for $\{ a \in \cP^{\fin}(A) \,\mid\, |a| \le m\}$.

\begin{definition}
A set $p \subseteq  \cP^{\fin} (A)$ is \emph{thick} if
$$\forall m \in \bbN\, \exists n \in \bbN \,\forall a \in \cP^{\le m} (A)\, \exists b \in p \cap \cP^{\le n} (A)  \,(a \subseteq b).$$
We let $\nu_p(m)$  denote the least $n$ with this property.
The set $p$ is \emph{thin} if it is not thick.
\end{definition}
Clearly $\{ a \in \cP^{\fin}(A) \,\mid\, x \in a \}$ is thick for
every $x \in A$ (with $\nu_p(m) = m+1$).  We now carry out the
developments of Subsections~\ref{SS1} and~\ref{extult} with
\emph{unbounded} and \emph{bounded} replaced by \emph{thick} and
\emph{thin}, respectively.  The definition of forcing and proofs of
Lemmas~\ref{basic1} and~\ref{fixf} are as before.  The following
observation enables the proof of Proposition~\ref{Los} to go through
as well.
\begin{lemma}
\label{union}
If $p$ is thick and $S \subseteq \cP^{\fin}(A)$, then either $p \cap
S$ or $p \setminus S$ is thick.
\end{lemma}

\begin{proof}
Otherwise there is $m_1$ such that 
\begin{equation} \forall n \in \bbN \,\exists a_1 \in \cP^{\le m_1} (A)\, \forall b \in  (p \cap S) \cap\cP^{\le n} (A)  \, (a_1 \nsubseteq b) \tag{*}
\end{equation}
and there is $m_2$ such that 
\begin{equation} \forall n \in \bbN \,\exists a_2 \in \cP^{\le m_2} (A)\, \forall b \in (p \setminus S) \cap \cP^{\le n} (A)   \, (a_2 \nsubseteq b).\tag{**}
\end{equation}
Let $m_1, m_2$ be as above, and let $m = m_1 + m_2$. Since $p$ is thick, 
\begin{equation}  \exists n \in \bbN \,\forall a \in \cP^{\le m} (A)\, \exists b \in p \cap \cP^{\le n} (A)  \,(a \subseteq b). \tag{***}\end{equation}
Fix such $n$; for $a_1 \in \cP^{\le m_1} (A)$ such that $\forall b \in (p \cap S)  \cap \cP^{\le n} (A) \, (a_1 \nsubseteq b)$ and  $a_2 \in \cP^{\le m_2} (A)$  such that $\forall b \in (p \setminus S)  \cap \cP^{\le n} (A)  \, (a_2 \nsubseteq b)$ we have
 $a_1 \cup a_2 \in \cP^{\le m} (A)$.  By (***) there is $b\in p \cap \cP^{\le n} (A)$ such that $a_1 \cup a_2 \subseteq b$. Depending on whether $b \in p \cap S$ or $b \in p\setminus S$, this contradicts (*) or (**).
\end{proof}

The next lemma enables a generalization of  Lemma~\ref{decideN}.  

\begin{lemma}\label{countableclosure}
Let $\la p_n \,\mid\, n \in \bbN \ra$ be such that, for all $n$, $p_n \in \bbP$ and $p_n \supseteq p_{n+1}$.
Then there is $p \in \bbP$ with the property that for every $n$ there is $k \in \bbN$ such that $\forall a \in (p \setminus p_n) \, ( |a| \le k)$. In particular, $p \setminus p_n$  is thin.
\end{lemma}

\begin{proof}
Define $p = \bigcup_{m=0}^{\infty}  \{ a \in p_m \,\mid\, |a| \le \nu_{p_m}(m)\}$.
Since $p \setminus p_n \subseteq \bigcup_{m=0}^{n-1}  \{ a \in p_m \,\mid\, |a| \le \nu_{p_m}(m)\}$,
$a \in p \setminus p_n$ implies $|a| \le k$ for $k =\max\{  \nu_{p_0}(0),\ldots,\nu_{p_{n-1}}(n-1)  \}$.

We show that $p$ is thick. Given $m \in \bbN$, we let $n = \nu_{p_m} (m)$.
If $a \in \cP^{\fin}(A)$ and $|a| \le m$, then there is $b \in p_m$ such that $a \subseteq b$ and 
$|b| \le n$. By the definition of $p$, $b\in p$. So $n$ has the required property.
\end{proof}

With these changes, the rest of the development of Subsections~\ref{SS1} and~\ref{extult} goes through and establishes the following strengthening of Proposition~\ref{ext}.

\begin{proposition} \label{bsptmodels}
Assume that $\cM \vDash `` I = \cP^{\fin} (A) \,\wedge\, A\text{ is infinite}$''.
The structure $\widehat{\fN}_A = (N_A, \in^{\ast} , M)$ constructed for this $I$ satisfies the principles of Transfer, Nontriviality, Boundedness, Standard Part, and Bounded Idealization over $A$ for $\in$-formulas with standard parameters.
\end{proposition}

\begin{proof}
To prove that $\widehat{\fN}_A \vDash \mathbf{O}$ one can take $d \in M$ such that 
$\cM \vDash ``d$ is a function on $ I=\cP^{\fin}(A) \,\wedge\, \forall a \in \cP^{\fin}(A)\, ( d(a) = |a| )$''.
Nontriviality also follows from Bounded Idealization.

 Let $G_m \in N$ and let $\la p, q \ra \in \cG$ have $\rank  q = k > m$.
There is some $X \in M$ such that $\cM \vDash ``\forall i \in I\, 
\forall \la x_0,\ldots, x_{k-1}\ra \in q(i)\,  (x_{m} \in X)$.''
By Proposition~\ref{Los} $\la p, q \ra \Vdash \forall v \, (v \in \dot{G}_m \imp v \in \check{X})$, so 
$\forall v \in G_m \, (v \in X)$ holds in $\fN_A $. This proves Boundedness in $\widehat{\fN}_A $.

It remains to prove that Bounded Idealization over $A$ holds in $\widehat{\fN}_A $.
Let $\phi(u,v)$ be an $\in$-formula with parameters from $M$. Assume that $\cM \vDash \forall a \in \cP^{\fin}(A)\, \exists y\, \forall x \in a\, \phi(x,y)$. 
By the Reflection Principle  in $\ZF$, $\cM \vDash`` \exists \text{ an ordinal } \alpha\,\, \forall a \in \cP^{\fin}(A) \, \exists
y \in V_{\alpha}\,\forall x \in a\, \phi(x,y)$''. We work in the model $\cM$.

Let $\la p, q\ra \in \bbH$ be a forcing condition and $k = \rank q$.
For $i = a\in p$ we let $q'(a) =$
$$ \{ \la x_0,\ldots, x_{k-1}, x_{k}\ra  \,\mid\, \la x_0,\ldots, x_{k-1}\ra \in q(a) \,\wedge \,  x_k \in V_{\alpha}\,\wedge\, \forall x \in a\, \phi(x,x_k)\};$$
 $q'(a) =\{\emptyset_{k+1}\}$ otherwise. Then $\la p, q' \ra $ extends $\la p, q \ra$.
For every $x \in A$ the set $c = p \cap \{ a \in \cP^{\fin}(A) \,\mid\, x \in a\} $ is in $\bbP$ because  $p \setminus c$ is thin (Lemma~\ref{union}),
 and so, by Proposition~\ref{Los}, $\la p, q' \ra \Vdash \phi(x,\dot{G}_k)$.

By the genericity of $\cG$ there exist  $\la p, q\ra \in \cG$ and $k \in \Omega$ such that  $\la p, q \ra \Vdash \phi(x,\dot{G}_k)$ for all $x \in A$. By the Fundamental Theorem~\ref{fteu} $\fN_A \vDash \phi(x,G_k)$ for all $x \in A$.
This means that the $\rightarrow$ implication of Idealization over $A$ for $\in$-formulas with standard parameters is satisfied in $\widehat{\fN}_A $. The opposite implication follows from $\SP$; see Lemma~\ref{downclosed} in Section~\ref{starzf}.
\end{proof}

\subsection{Further theories}
Nelson's $\IST$ postulates a form of Idealization that is even stronger than Bounded Idealization (but it contradicts Boundedness).

$\mathbf{I}$ (Idealization) \quad Let $\phi$ be an $\in$-formula with arbitrary parameters.
$$\forall^{\st\fin}  a \;\exists y\; \forall x \in a\;   \phi(x,y)\eqi 
\exists y \; \forall^{\st} x  \; \phi(x,y).$$

The theory $\BSPT' = \ZF + \textbf{T} + \textbf{SP}' +\B +  \textbf{BI}'$ is introduced in Subsection~\ref{s13}.
We let $\ISPT'$ be the theory obtained from $\BSPT'$ by deleting Boundedness and replacing Bounded Idealization for formulas with standard parameters by  $\mathbf{I}'$, Nelson's Idealization for formulas with standard parameters.
In other words, $\ISPT'= \ZF + \textbf{T} + \textbf{SP}' + \textbf{I}'$.
\label{ispt'}

Principle $\mathbf{I}$ implies the existence of a finite set that contains all standard sets as elements, and has certain undesirable consequences from the metamathematical point of view. Kanovei and Reeken~\cite{KR}, Theorem 4.6.23,  prove that there are countable models $\cM = (M, \in^{\cM})$ of $\ZFC$ that cannot be extended to a model of $\IST$ in which $M$ would be the class of all standard sets (assuming $\ZFC$ is consistent). We do not know whether the same is the case for $\ISPT'$. Nevertheless we have the following result.

\begin{theorem}\label{ispt}
$\ISPT'$ is a conservative extension of $\ZF$.
\end{theorem}

\begin{proof}
Let us assume that  $\ISPT' \vdash \theta$ but $\ZF \nvdash \theta$, for some $\in$-sentence $\theta$.
Let $\ZF_r$ be $\ZF$ with the Axiom Schema of Replacement restricted to $\Sigma_r$-formulas, 
and let $\ISPT'_r$ be $\ISPT'$ with $\ZF$ replaced by $\ZF_r$.
There is $r \in \omega$ for which $\ISPT'_r \vdash \theta$.

Let $\cM = (M, \in^{\cM})$ be a model of $\ZF + \neg \theta$.
By the Reflection Principle of $\ZF$, valid in $\cM$, there is $\alpha \in M$ such that 
$\cM \vDash ``\alpha \text{ is a limit ordinal}$'', $\cM \vDash \phi^{V_{\alpha}}$  for every axiom $\phi$ of $\ZF_r$, and $\cM \vDash (\neg \theta)^{V_{\alpha}}$.

We let $A = (V_{\alpha})^{\cM}$ and  use Proposition~\ref{bsptmodels}  to extend $\cM$ to a model $\widehat{\fN}_A$.  We define
$N_{\alpha} = \{ x \in N_A \,\mid\, \widehat{\fN}_A \vDash x \in V_{\alpha}\}$.
It is easy to verify that $\widehat{\fN}_{\alpha} = (N_{\alpha}, \in^{\ast}\, \uh \! N_{\alpha}, M \cap N_{\alpha})$ is a model of $\ISPT'_r $;
hence $\theta$ holds in $\widehat{\fN}_{\alpha} $. On the other hand, $(\neg \theta)^{V_{\alpha}}$ holds in $\cM$ and hence, by Transfer, $\neg \theta$ holds in $\widehat{\fN}_{\alpha} $. A contradiction.
\end{proof}

Kanovei and Reeken~\cite{KR}, Theorem 3.4.5, showed that the class of bounded sets in $\IST$ satisfies the axioms of $\BST$. 
This result holds also for $\ISPT'$ and $\BSPT'$, respectively,  and establishes the following theorem.

\begin{theorem}\label{bspt}
The theory $\BSPT'$ is a conservative extension of $\ZF$.
\end{theorem}

\emph{This concludes the proof of Theorem \textbf{C}.}

\bigskip
Finally, we prove that if $\cM$ satisfies $\ADC$, then the model  $\widehat{\fN}_A$ constructed in   Proposition~\ref{bsptmodels} satisfies $\CC$. 
We note that the definition of forcing for $\st$-$\in$-formulas in Section~\ref{forcingexternal},  and Lemma~\ref{homogeneity}, extend to $I = \cP^{\fin} (A)$.

\begin{proposition}\label{exxt} If $\cM$ is a countable model of $\ZFc$, then
the extended ultrapower $\widehat{\fN}_A$ satisfies Countable $\st$-$\in$-Choice (both $\CC$ and $\CC^{\st}$). 
\end{proposition}

\begin{proof}
We work in $\ZFc$.

$\CC$: Let us assume that $\la p, q \ra $ has rank $k$ and $\la p, q \ra \Vdash \forall^{\st} m \in \bbN\, \exists v \,\phi (m, v)$.
Let $E_m = $
$$\{ \la p', q' \ra  \in \bbH \,\mid\, \la p', q'\ra \le  \la p, q \ra \,\wedge\,
 \la p', q'\ra  \Vdash \phi (m, \dot{G}_n) \text{ for some } n> k\}.$$
By an argument like the one in Lemma~\ref{forall}  it follows that for every $m$ and every $\la p'', q''\ra \le  \la p, q \ra $ there is $ \la p', q' \ra \le  \la p'', q'' \ra $ such that $ \la p', q' \ra \in E_m$.
[We note that the $E_m$ may be proper classes, but by the Reflection Principle  there is a set $S$ such that $\la p, q \ra \in S$ and for every $m$ and every $\la p'', q''\ra \in S$ there is $ \la p', q' \ra \le  \la p'', q'' \ra $ such that $ \la p', q' \ra \in S \cap E_m$. The classes $E_m$ can be replaced by the sets $S \cap E_m$ in the argument below.]

We let 
$\la m', \la p' , q' \ra \ra \mathbf{R} \la m'', \la p'', q''\ra \ra\,\text{  iff }$
$$ \la p'', q'' \ra \le \la p', q'\ra \le \la p, q \ra\,\wedge\, m''= m'+1 \,\wedge \, \la p' , q' \ra \in E_{m'}\,\wedge\, \la p'', q''\ra \in E_{m''}.$$

Applying $\ADC$ to the relation $\mathbf{R}$ we obtain  a sequence $ \la \la p_m, q_m \ra \,\mid\, m \in \bbN \ra$ 
such that $ \la p_0 , q_0 \ra \le \la p , q \ra $ and, for each $m$,   
 $\la p_{m+1}, q_{m+1}\ra \le \la p_m, q_m\ra$ and $\la p_m, q_m \ra \Vdash \phi (m, \dot{G}_n)$ for some $n \in \bbN$, $n > k$. Let $\rank q_m = \ell_m$ and let $n_m < \ell_m$, $n_m > k$,  be the least such $n$.

As in the proof of Lemma~\ref{countableclosure}, let
$p_{\infty} = \bigcup_{m=0}^{\infty} C_m$ where $C_m = \{ a \in p_m  \,\mid\, |a| \le \nu_{p_m}(m)\}$. We recall that 
$p_{\infty} \setminus p_m$ is thin for every $m$;
hence $\la p_{\infty}, q_m\ra \Vdash \phi (m, \dot{G}_{n_m})$.
We define a function $q_{\infty} \in \bbQ$ of rank $k+1$ as follows: 
If $a \in C_m \setminus \bigcup_{j=0}^{m-1} C_{j}$ then   
$$q_{\infty}(a) = \{ \la x_0,\ldots, , x_k\ra \,\mid\, x_k \text{ is a function } \,\wedge\,\dom x_k = \bbN \, \wedge \forall j \le m $$ 
$$\exists \,  y_k,\ldots, y_{\ell_{j}-1}  \; 
(\la x_0,\ldots, x_{k-1},  y_k,\ldots, y_{\ell_{j}-1}\ra \in q_j(a) \,\wedge\, x_k(j) = y_{n_j} )\, \wedge\,$$ $$\forall j > m\, (x_k(j) = 0)\}  .$$
By ``\L o\'{s}'s Theorem'', $\la p_{\infty}, q_{\infty}\ra \Vdash ``\dot{G}_k \text{ is a function with } \dom \dot{G}_k = \bbN$.''

Now assume that $\widehat{\fN}_A \vDash \forall^{\st} m \, \exists v \,\phi (m, v)$.
Then there is $\la p, q \ra \in \cG$ such that $\la p, q \ra \Vdash \forall^{\st} m \, \exists v \,\phi (m, v)$.
By the above discussion, there is a condition of the form $\la p_{\infty}, q_{\infty} \ra$ such that $\la p_{\infty}, q_{\infty} \ra \in \cG$; hence 
$\widehat{\fN}_A \vDash `` G_k \text{ is a function with } \dom G_k = \bbN$.''
Fix $m \in M$; let $\widehat{\fN}_A \vDash G_k (m) = G_{\ell}$; we can assume $\ell > k$.
Then there is some  $\la p', q' \ra \in \cG$,  $\la p', q' \ra \le \la p_{\infty}, q_{\infty} \ra$, such that 
$\la p', q' \ra \Vdash \dot{G}_k (m) = \dot{G}_{\ell}$.

Let $\sigma_m = k \cup \{ n_m\}$ and $\sigma' = k \cup \{ \ell \}$.

From $\la p_{\infty}, q_m \ra \Vdash \phi(m,\dot{G}_{n_m}) $
and Lemma~\ref{homogeneity} it follows that the condition
$\la p_{\infty}, q_m \uh \sigma_m  \ra \Vdash \phi(m,\dot{G}_{k}) $.
We see from the construction that $\la p'\cap p_m, q' \uh \sigma' \ra \le \la p'\cap p_m, q_{m} \uh \sigma_m\ra$.
As $p' \setminus p_m$ is thin, we have $\la p', q'  \uh \sigma'\ra \Vdash \phi(m,\dot{G}_{k})$, 
and by Lemma~\ref{homogeneity} again,
$\la p', q' \ra \Vdash \phi(m,\dot{G}_{\ell}) $.
From $\la p', q' \ra \in \cG$, we conclude that 
$\widehat{\fN}_A \vDash  \phi(m,G_{\ell})$ and hence also 
$\widehat{\fN}_A \vDash  \phi(m,G_{k}(m))$.

$\CC^{\st}$:
Assuming $\la p, q \ra \Vdash \forall^{\st} m \in \bbN\, \exists^{\st} v \,\phi (m, v)$, we can require in the definition of $E_m$ that 
$\la p', q'\ra  \Vdash \st(\dot{G}_n)$, i.e.,  $\la p', q'\ra  \Vdash \dot{G}_n = c_{z_n}$ for a uniquely determined $z_n$. In the definition of $q_{\infty} (a) $ we can let $x_k = \la  z_{n_j}  \,\mid \, j \in \bbN\ra$.
Then $\la p_{\infty}, q_{\infty} \ra \Vdash \st( \dot{G}_k)$ and $\widehat{\fN}_A \vDash \st(G_k)$.
\end{proof}

Let $\BSCT'$ be the theory obtained from $\BSPT'$ by adding $\ADC$ and strengthening $\SP$ to $\CC$; analogously for $\ISCT'$.
The last two theorems of this section follow by the same arguments as those used to prove Theorems~\ref{ispt} and~\ref{bspt}.

\begin{theorem}
The theory $\ISCT'$ is a conservative extension of $\ZFc$.
\end{theorem}

\begin{theorem}
The theory $\BSCT'$ is a conservative extension of $\ZFc$.
\end{theorem}


\section{Final Remarks}
\subsection{Open problems}$\!$

(1)  Are the theories $\BSCT' + \SN$  and $\ISCT' + \SN$ (defined above) conservative extensions of $\ZFc$?

We do not know whether $\widehat{\fN}_A$ for $I = \cP^{\fin}(A)$ with uncountable $A$ satisfies $\SN$. In the absence of $\AC$, a way to formulate and prove a suitable analog of Lemma~\ref{isomorphism} is not obvious.

(2) Are the theories $\BSPT$ and $\ISPT$ conservative extensions of $\ZF$?
Are the theories $\BSCT$ and $\ISCT$ conservative extensions of $\ZFc$?

Here $\BSPT$ is obtained from $\BSPT'$ by strengthening (Bounded) Idealization to allow arbitrary parameters; similarly for the other theories.
The likely answer is yes; the obvious approach is to iterate the forcing used to prove the primed versions.
Spector develops iterated extended ultrapowers in~\cite{Spr2}. His method would require nontrivial adaptations in our framework, but it is likely to work provided the answer to problem (1) is yes.
The ultimate result would be that $\BSCT + \SN$ and $\ISCT + \SN$ are conservative extensions of $\ZFc$.

(3) Does every countable model of $\ZF$ have an extension to a model of $\BSPT'$?

The likely answer is again yes, using a suitable iteration of extended ultrapowers. 

(4) Is $\SPOT + \SC$ a conservative extension of $\ZF$?

\subsection{Forcing with filters}
A  more elegant and potentially more powerful notion of forcing is obtained by replacing $\bbP$ with 
$$\widetilde{\bbP} = \{ \cP \,\mid\, \cP\text{ is a filter of unbounded subsets of } I \}$$
where $\cP'$ extends $\cP$ iff $\cP \subseteq \cP'$. ``\L o\'{s}'s Theorem''~\ref{Los} then takes the form: \quad
$\la \cP, q \ra  \Vdash  \phi(\dot{G}_{n_1},\ldots,\dot{G}_{n_s})$ iff
$\rank q  = k > n_1,\ldots, n_s$ and \\
$\exists p \in \cP\, \forall i \in p \, \forall \la x_0,\ldots, x_{k-1}\ra \in q(i)\; 
\phi(x_{n_1},\ldots, x_{n_s}).$\\
The forcing notion $\bbP$ we actually use amounts to restricting oneself to principal filters.

\subsection{Zermelo set theory} 
Similar results can be obtained for theories weaker than $\ZF$. 
Let $\textbf{Z} = \ZF - \text{Replacement}$ be the Zermelo set theory, and let
$\textbf{BT}$ denote Transfer for bounded formulas. 
In the proof of Proposition~\ref{countablemodels} the extended ultrapower can be replaced by the \emph{extended bounded ultrapower} (see~Chang and Keisler~\cite{CK}, Sec. 4.4, for a discussion of ordinary bounded ultrapowers).  This proves that 
$\SPOT^{-} =  \textbf{Z}  + \N + \textbf{BT} +\SP$ \emph{is a conservative extension of} $\textbf{Z}$.
With some modifications, this theory can be taken as an axiomatization of the \emph{internal part} of  \emph{nonstandard universes} of Keisler~\cite{CK, K2} (the superstructure framework for nonstandard analysis).
Analogous results can be obtained for $\SCOT^{-}$, $\BSPT^{-}$ and $\BSCT^{-}$.

\subsection{Weaker theories} 
\emph{Reverse Mathematics} has as its goal the calibration of the exact set-theoretic strength of the principal results in ordinary mathematics. 
One of its chief accomplishments is the discovery that, with a few exceptions, every theorem in ordinary mathematics is logically equivalent (over $\textbf{S}_1$) to one of the five subtheories 
$\textbf{S}_1 - \textbf{S}_5$ 
of second-order arithmetic $\textbf{Z}_2$ (Simpson~\cite[p.\;33]{Si}), known collectively as ``The Big Five.''
Here $\textbf{S}_1$ is the weakest of the five theories, the second-order arithmetic with recursive comprehension axiom, also denoted $\textbf{RCA}_0$, and $\textbf{S}_2$, known as $\textbf{WKL}_0$,
is obtained by adding the Weak K\"{o}nig's Lemma to the axioms of $\textbf{RCA}_0$.
 We refer to Simpson~\cite{Si} for a comprehensive introduction to Reverse Mathematics.

Keisler and others extended the ideas of Reverse Mathematics to the nonstandard realm.  
In Keisler's paper~\cite{K3} it is established that if $\textbf{S}$ is any of the ``Big Five'' theories above, then $\textbf{S}$ has a conservative extension ${}^{\ast}\textbf{S}$ to a theory in the language with an additional unary predicate $\st$; the axioms of ${}^{\ast}\textbf{S}$ include $\textbf{O}$, 
$\textbf{SP}$ and, for the theories stronger than $\textbf{WKL}_0$, also $\textbf{FOT}$ (First-Order Transfer).

In a somewhat different direction, there is an extensive body of work by van den Berg, Sanders and others (see~\cite{VBS}, \cite{Sa} and the references therein) devoted to determining the exact proof-theoretic strength of particular results in infinitesimal ordinary mathematics.
Substantial parts of it can be carried out in these and other elementary systems for nonstandard mathematics,  for example Nelson~\cite{N} and  Sommer and Suppes~\cite{SS}. 
However, these systems do not enable the natural reasoning as practiced in analysis.  
They are usually formalized in the language of second-order arithmetic or type theory.
Basic objects of ordinary analysis, such as real numbers, continuous functions and separable metric spaces, have to be represented in these theories via suitable codes, and  the results may have to be presented ``up to infinitesimals,'' because the full strength of the Transfer principle or the Standard Part principle is not available.
The focus of this paper is on theories like $\SPOT$ or $\SPOT^{-}$, which axiomatize both the traditional and the nonstandard methods of ordinary mathematics in the way they are customarily practiced.
 Rather than looking for the weakest principles that enable a proof of a given mathematical theorem, we formulate theories that are as strong as possible while still \emph{effective} (conservative over $\ZF$) or \emph{semi-effective} (conservative over $\ZFc$).

\subsection{Finitistic proofs}
\label{finitistic}
The model-theoretic proof of Proposition~\ref{countablemodels} as given here is carried out in $\ZF$. 
Using techniques from Simpson~\cite{Si}, Chapter II, esp. II.3 and II.8, it can be verified that the  proof goes through in $\textbf{RCA}_0$ (w.l.o.g. one can assume that $M \subseteq \omega$).

The proof of Theorem \textbf{A} from Proposition~\ref{countablemodels} requires the G\"{o}del Completeness Theorem and therefore $\textbf{WKL}_0$; see~\cite{Si}, Theorem IV.3.3.
We conclude that Theorem \textbf{A} can be proved in $\textbf{WKL}_0$. 

Theorem \textbf{A}, when viewed as an arithmetical statement resulting
from identifying formulas with their G\"{o}del numbers, is
$\Pi^{0}_{2}$.  It is well-known that $\textbf{WKL}_0$ is conservative
over $\textbf{PRA}$ (Primitive Recursive Arithmetic) for $\Pi^{0}_{2}$
sentences (\cite{Si}, Theorem IX.3.16); therefore Theorem\;\textbf{A}
is provable in $\textbf{PRA}$.  The theory $\textbf{PRA}$ is generally
considered to correctly capture finitistic reasoning as envisioned by
Hilbert~\cite{Hi} (see e.g., Simpson~\cite{Si}, Remark IX.3.18) and
Hilbert--Bernays \cite{Hi34, Hi39} (see Zach~\cite{Za07}, p.\;417).
We conclude that Theorem \textbf{A} has a finitistic proof.

These remarks apply equally to Theorems \textbf{B} - \textbf{D}.

\subsection{$\SPOT$ and $\CH$}
Connes (see for example~\cite{C}, pp.\;20--21) objects to the use of ultrafilters but approves of the Continuum Hypothesis ($\CH$).
In the absence of full $\AC$, it is important to distinguish (at least) two versions of $\CH$.

$\CH$:  \quad Every infinite subset of $\bbR$ is either countable or equipotent to $\bbR$.

$\CH^{+}$:  \; $\bbR$ is equipotent to $\aleph_{1}$ (often written $2^{\aleph_0} = \aleph_1$).

The axioms $\CH$ and $\CH^{+}$ are equivalent over $\ZFC$, but not over $\ZFc$.
It is known that  $\ZFc + \CH$  does not imply the existence of any nonprincipal ultrafilters over $\bbN$ ($\CH$ holds in the Solovay model).  We have:

\begin{proposition}
The theory $\SPOT+ \CH$ is a conservative extension of $\ZF+\CH$.
\end{proposition}

\begin{proof}   Let $\phi$ be an $\in$-sentence. Then
$\SPOT + \CH \vdash \phi$ iff $\SPOT \vdash (\CH \imp \phi)$ iff $\ZF
  \vdash ( \CH \imp \phi)$ iff $\ZF + \CH \vdash \phi$. 
\end{proof}

However, it seems clear that Connes has $\CH^{+}$ in mind. 
But $\CH^{+}$ implies that $\bbR$ has a well-ordering (of order type $\aleph_1$).
From this it easily follows that there exist nonprincipal ultrafilters over $\bbN$
(for example, Jech~\cite{J2}, p. 478 proves a much stronger result).
Thus Connes's position on this matter is incoherent.

Apart from the issue of $\CH$, Connes's repeated criticisms of
Robinson's framework starting in 1994 are predicated on the premise
that infinitesimal analysis requires ultrafilters on $\bbN$ (which are
incidentally freely used in some of the same works where Connes
criticizes Robinson).  Our present article shows that Connes's premise
is erroneous from the start.%
\footnote{A more detailed analysis of Connes’s views can be found in
  Sanders (\cite{Sa20}, 2020) and references therein.}

\subsection{External sets}\label{externalsets}
This paper employs only definable external sets, and only in Subsection~\ref{scot}.
It is sometimes claimed that the axiomatic approach is inferior to the model-theoretic one because substantial use of external sets is essential for some of the most important new contributions of Robinsonian nonstandard analysis to mathematics, such as the constructions of nonstandard hulls and Loeb measures. 
The following observations are relevant:
\begin{itemize}
\item
Except in some very special cases, nonstandard hulls and Loeb measures fall in the scope of set-theoretic mathematics, and the use of $\AC$ in their construction is not an issue.
\item
Hrbacek and Katz~\cite{HK} demonstrate that nonstandard hulls and Loeb measures can be constructed in internal-style nonstandard set theories such as $\BST$ and $\IST$.
\item
The theory $\BST$ can seamlessly be extended to $\HST$, a nonstandard set theory that axiomatizes also external sets (see Kanovei and Reeken~\cite{KR}).
In $\HST$ the constructions of nonstandard hulls and Loeb measures can be carried out in ways analogous to those familiar from the model-theoretic approach.
\end{itemize}

\section{Conclusion}  
In this paper we establish that infinitesimal methods in ordinary mathematics require no Axiom of Choice at all, or only those weak forms of $\AC$ that are routinely used in the traditional treatments. 
This conclusion follows from the fact that the theory $\SPOT$ and its various strengthenings, which do  not imply the existence of nonprincipal ultrafilters over $\bbN$, or other strong forms of $\AC$, are sufficient to carry out infinitesimal arguments in ordinary mathematics (and beyond).

But most users of nonstandard analysis work with hyperreals, and the existence of hyperreals does imply the existence of nonprincipal ultrafilters over $\bbN$. 
So it would seem that ultrafilters are needed, after all.
However, this view implicitly assumes that set theory like $\ZFC$, based exclusively on the membership predicate $\in$, is the only correct framework for the Calculus.

Historically,%
\footnote{See for example Katz and Sherry~\cite{KS} and Bair et al.\;\cite{20z}.}
the Calculus of Newton and Leibniz was first made rigorous by Dedekind, Weierstrass and Cantor in the 19th century  using the $\varepsilon$-$\delta$ approach. It was eventually axiomatized in the $\in$-language as $\ZFC$.
After Robinson's development of nonstandard analysis it was realized  that Calculus with infinitesimals  also admits a rigorous formulation, closer to the ideas of Leibniz, Bernoulli, Euler (see~\cite{BBE}) and Cauchy (see~\cite{BBF}). It can be axiomatized in a 
%
set theory using the $\st$-$\in$-language. The primitive predicate $\st$ can be thought
of as a formalization of the Leibnizian distinction between 
assignable and inassignable quantities.
Such  theories are obtained from $\ZFC$ by adding suitable versions of Transfer, Idealization and Standardization.

Now that it has been established that the infinitesimal methods do not
carry a heavier foundational burden than their traditional
counterparts, one can ask the following question.  Which foundational
framework constitutes a more faithful formalization of the
techniques of the 17--19 century masters?  For all the achievements of
Cantor, Dedekind and Weierstrass in streamlining analysis, built into
the transformation they effected was a failure to provide a theory of
infinitesimals which were the bread and butter of 17--19 century
analysis, until Weierstrass.  By the yardstick of success in
formalization of classical analysis, arguably $\SPOT$, $\SCOT$ and
other theories developed in the present text are more successful
 than $\ZF$ and $\ZF+\ADC$.

One can learn to work in the universe of an $\st$-$\in$-set theory
intuitively.  This universe can be viewed as an extension of the
standard set-theoretic universe, either by a (soritical) predicate
$\st$ (the \emph{internal picture}) or by new ideal objects (the
\emph{standard picture}); see Fletcher et al.~\cite{F}, Sec.~5.5, for
a detailed discussion.%
\footnote{The internal picture fits well with the multiverse
philosophy of Hamkins~\cite{Ha12}.  One of his postulates is
Well-foundedness Mirage: Every universe $V$ appears to be ill-founded
from the point of view of some better universe (\cite{Ha12},
Sec. 9). The internal picture proposes something stronger but in the
same spirit: the ill-foundedness is witnessed by a predicate $\st$ for
which $(V, \in, \st)$ satisfies $\BST$.  This point is elaborated
in~\cite{F}, Section 7.3.}
This extended universe has a unique set of real numbers, constructed
in the usual way, and containing both standard and nonstandard
elements.  It also of course has choice functions and ultrafilters
over $\bbN$, just as the universe of $\ZFC$ does.  Mathematicians
concerned about $\AC$ can analyze their methods of proof and determine
whether a particular result can be carried out in one of the theories
considered in this paper.  For most if not all of ordinary
mathematics, both traditional and infinitesimal, the answer is likely
to be affirmative. It then follows from Theorems \textbf{A} -
\textbf{D} that these results are just as effective as those provable
in respectively $\ZF$ or $\ZF+\ADC$.



\bibliographystyle{amsalpha}
\end{document}